\documentclass[journal,twoside,web]{ieeecolor}

\usepackage{graphicx,color}
\usepackage{amsmath,amssymb,amsfonts}
\usepackage{mathrsfs}
\usepackage{subfigure}
\usepackage{url}
\usepackage{booktabs}
\usepackage{array}
\usepackage[table]{xcolor}
\usepackage{mathtools} 		
\usepackage{cite}
\usepackage[vlined,ruled,linesnumbered]{algorithm2e}
\usepackage{upgreek}
\usepackage{dsfont}
\usepackage[figurename=Fig.,font=footnotesize]{caption}
\usepackage{calc}
\usepackage{generic}
\usepackage{hyperref}
\hypersetup{hidelinks=true}
\usepackage{textcomp}

\graphicspath{{./figures/}}

 \usepackage{epstopdf}
\def\BibTeX{{\rm B\kern-.05em{\sc i\kern-.025em b}\kern-.08em
    T\kern-.1667em\lower.7ex\hbox{E}\kern-.125emX}}

\usepackage{amsthm}
\theoremstyle{definition}
\newtheorem{theorem}{Theorem}  

\newtheorem{definition}{Definition}

\newtheorem{proposition}[theorem]{Proposition}
\newtheorem{problem}{Problem}
\newtheorem{remark}{Remark}
\newtheorem{assumption}{Assumption}
\newtheorem{example}{Example}

\newcommand{\mc}{\mathcal}

\newcommand{\im}{\operatorname{Im}}

\newcommand{\real}{\mathbb{R}}

\newcommand{\naturalpos}{\mathbb{N}_{>0}}
\newcommand{\naturalnneg}{\mathbb{N}_{\geq 0}}
\newcommand{\realpos}{\mathbb{R}_{> 0}}
\newcommand{\realnneg}{\mathbb{R}_{\geq 0}}

\newcommand{\complex}{\mathbb{C}}
\newcommand{\tsp}{\mathsf{T}}
\newcommand{\pinv}{\dagger} 
\newcommand{\inv}{{\negat 1}} 
\newcommand{\negat}{\scalebox{0.75}[.9]{\( - \)}}
\newcommand*{\QEDB}{\hfill\ensuremath{\square}}

\newcommand*{\QEDBL}{\hfill\ensuremath{\blacksquare}}

\usepackage{pifont} 
\renewcommand{\ast}{\text{\ding{86}}}

\newcommand{\map}[3]{#1: #2 \rightarrow #3}

\newcommand{\norm}[1]{\Vert #1 \Vert}

\DeclareMathAlphabet{\mymathbb}{U}{BOONDOX-ds}{m}{n}

\usepackage{colonequals}

\def\symmetric{\mathbb{S}}

\newcommand{\bmat}[1]{\begin{bmatrix}#1\end{bmatrix}}

\newcommand{\blue}[1]{#1}

\begin{document}

\title{The Internal Model Principle of \\ Time-Varying Optimization}

\author{Gianluca Bianchin, \IEEEmembership{Member, IEEE}, and Bryan Van Scoy, \IEEEmembership{Member, IEEE}
\thanks{This material is based upon work supported in part by the National
Science Foundation under Award No. 2347121 and the FRFS
WEL-T Investigator Programme. Any opinions, findings and conclusions or
recommendations expressed in this material are those of the authors and do
not necessarily reflect the views of the National Science Foundation.}%
\thanks{%
G.~Bianchin is with the ICTEAM institute and the Department of Mathematical Engineering (INMA) at the University of Louvain, Belgium.
Email: \texttt{gianluca.bianchin@uclouvain.be}}%
\thanks{
B.~Van~Scoy is with the Department of Electrical and Computer Engineering, Miami University, Oxford, OH~45056, USA. Email: \texttt{bvanscoy@miamioh.edu}%
}}

\maketitle

\begin{abstract}
Time-varying optimization problems are central to many engineering 
applications, where performance metrics and system constraints evolve 
dynamically with time. 
Several algorithms have been proposed to address these problems; a common 
characteristic among them is their implicit reliance on knowledge of the 
optimizers’ temporal variability.
In this paper, we provide a fundamental characterization of this 
property: we show that an algorithm can track time-varying optimizers if and 
only if it incorporates a model of the temporal variability of the optimization 
problem. We 
refer to this concept as the \textit{internal model principle of time-varying 
optimization.} 
Our analysis relies on showing that designing optimization algorithms for time-varying problems can be reformulated as an output regulation problem. By leveraging tools from center manifold theory, we derive 
necessary and sufficient conditions for exact asymptotic tracking. These results, in turn, 
pave the way for the development of new optimization algorithms.
We demonstrate the effectiveness of the approach through numerical experiments 
on both synthetic problems and the dynamic traffic assignment problem 
from traffic control.
\end{abstract}


%
\section{Introduction}
\IEEEPARstart{T}{ime-varying} optimization problems play a central role in several scientific 
domains, as they underpin many important contemporary engineering problems. 
Examples include training in Machine Learning~\cite{YC-YZ:20,AR-KS:13}, 
dynamic signal estimation in Signal Processing~\cite{FJ-AR:12}, trajectory 
tracking in Robotics~\cite{AD-VC-AG-GR-FB:23}, system optimization in 
Industrial Control~\cite{CB-BS-BD:09}, and much more. 
Historically, discrete-time algorithms for time-varying optimization
have been proposed and studied first since they emerged as a natural extension 
of their time-invariant counterparts, allowing for cost functions and 
constraints that may change over time~\cite{VZ-MA:10}.
These approaches build on the classical perspective on 
optimization, which seeks to construct methods to determine optimizers and 
consist of iterative procedures implemented on digital devices.
In recent years, however, numerous optimization algorithms have been 
formulated and studied as continuous-time 
processes~\cite{WS-SB-EC:16,AC-EM-JC:16,MM-MJ:19}, mainly due to the wide 
availability of tools for Ordinary Differential Equations (ODEs) that can 
facilitate their analysis~\cite{MM-MJ:19}.

\begin{figure}[tb]
\centering
\includegraphics[scale=0.85]{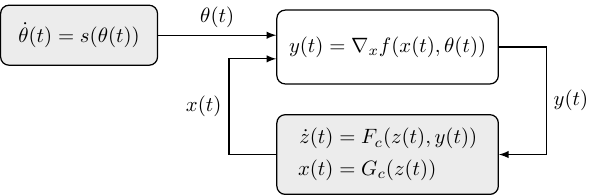}
\caption[]{Architecture of the gradient-feedback design scheme studied in this 
work. An optimization algorithm is to be designed (bottom block), having access to gradient evaluations of the loss to be minimized (top right 
block), and generating a sequence of exploration points $x(t)$ at which the 
gradient shall be evaluated.
The loss function varies with time, where the temporal 
variability $\theta(t)$ is assumed to be unmeasurable and generated by an 
exosystem (top left block). Shaded blocks emphasize the presence of dynamics.}
\label{fig:feedback_design}
\vspace{-.5cm}
\end{figure}

Motivated by these recent developments, in this paper, we study time-varying 
convex optimization problems and focus on the use of continuous-time dynamics 
to track exactly optimal solutions.
Several approaches have already been developed for this 
purpose~\cite{YZ-MS:98,MF-SP-VP-AR:17,bastianello2024,AS-ED-SP-GL-GG:20}, yet all 
these techniques implicitly require full knowledge of the temporal variability 
of the problem~\cite{AS-ED-SP-GL-GG:20}. Unfortunately, in most 
practical 
applications, having such knowledge is impractical, either because the temporal 
variability enters the optimization in the form of exogenous disturbances that 
are unknown and cannot be measured (see, e.g., \cite{GB-JC-JP-ED:21-tcns,AS-ED-JM-AB:21}), or simply because it is unrealistic to ask for a noiseless model of the problem's variability. 
Departing from this, in this paper we pose the following question: {\it is 
it possible to track (exactly and asymptotically) a minimizer of a 
time-varying  optimization problem without any knowledge of the temporal variability of the  optimization?}
Interestingly, we prove a fundamental result showing that tracking can be 
achieved if and only if the temporal variability of the problem 
can be `observed' by the algorithm, and the latter incorporates a suitably 
reduplicated model of such a variability.
We refer to this conclusion as \textit{the internal model principle of time-varying 
optimization,} akin to its control-theoretic counterpart~\cite{BF-WW:76,JH-WW:84}.
Our approach relies on reinterpreting the optimization algorithm design as a 
nonlinear, multivariable regulation problem~\cite{BF:77}, and our analysis 
uses tools from center manifold theory~\cite{JC:81,JH-WW:84}.~Fig.~\ref{fig:feedback_design} illustrates the architecture of the 
gradient-feedback design scheme studied in this paper.

\textit{Related works.}
The literature on methods for time-varying optimization is mainly divided into 
two classes of solutions. 
The first class consists of methods that do not utilize any 
model of the temporal variability of the 
problem~\cite{EH-RW:15,SF:20,SS:12,RD-AB-RT-KR:19}; 
instead, they seek to solve a sequence of static problems. 
Several established approaches belong to this class, including the online 
gradient descent method~\cite{MZ:03} and the online Newton step 
algorithm~\cite{EH-AA-SK:07}---see~\cite{EH:16} and references therein.
Because the temporal variability of the problem is unknown (or ignored), 
algorithms in this class can make a decision only \textit{after} each 
variation has been observed, thus incurring a certain {\it regret}.
Mathematically, these techniques are capable of reaching only a 
neighborhood of an optimizer~\cite{AS-ED-SP-GL-GG:20}, and exact tracking 
is out of reach in general for these approaches. \blue{This may not be surprising, as it has been recognized recently that these algorithms can be interpreted as basic integral controllers~\cite{IN-MB-LM-GN:23}.}
Moreover, suitable 
assumptions need to be made to provide convergence guarantees for these 
methods; we refer to~\cite{AS-ED-SP-GL-GG:20,EH:16} and references therein 
for a detailed discussion.
In contrast, the second class of methods uses a model of the temporal evolution 
of the problem to exactly track the optimal 
trajectories~\cite{YL-GQ-NL:18,AR-KS:13,AJ-AR-SS-KS:15,OD-NH-PJ:17}. 
Particularly celebrated is the prediction-correction algorithm 
(see~\cite{YZ-MS:98}, the recent work~\cite{MF-SP-VP-AR:17}, and the discrete implementations~\cite{AS-AM-AK-GL-AR:16}), 
whereby, at each time, a prediction step is used to anticipate how the optimizer 
will evolve over time, and a correction step is used to seek a solution to 
each instantaneous optimization problem. 
We refer to~\cite{AS-ED-SP-GL-GG:20} for a recent overview of the topic. 
Recent years have witnessed a growing interest in this problem: 
\cite{AD-VC-AG-GR-FB:23} uses contraction to study these methods; 
\cite{MM-JB-JS-PT:24} uses sampling to estimate the temporal variability of 
the problem; a recent survey has appeared in~\cite{AH-ZH-SB-GH-FD:24}; 
constrained optimization problems are studied in~\cite{AS-ED:17}; recent 
works~\cite{RR-AM-UV:22,YS-CC-XZ-WC:24,RR-AM-UV:23,KG-DP:20} have modified 
these methods to achieve fixed-time convergence.
With respect to this body of literature, 
which assumes precise knowledge of the temporal variability of the problem
(see detailed discussion in Section~\ref{sec:prob_formulation_a}), our focus 
here is to answer the following fundamental question: what is the 
least-restrictive information that is needed to track, asymptotically with zero 
error, the optimizer of a time-varying problem? To the best of the authors' 
knowledge, a rigorous answer to this question is missing in the literature.

In this work, we recast the problem of designing optimization algorithms for 
time-varying optimization as a nonlinear output regulation problem. This 
connects our work with the literature on output regulation, which 
is a well-established area of research. Initial works can be traced 
back to the 1970s~\cite{BF-WW:76,ED:76,BF:77} focusing on linear systems, and 
later extended to nonlinear systems using both 
local~\cite{AI-CIB:90,CB-FP-AI-WK:97} and global~\cite{CB-AI-LP:03} approaches 
for the analysis. In recent years, the field has received new attention using 
modern methods; see, e.g., \cite{LW-AI-HS-LM:16,AI-LM-LP:12,JH:04} and the 
recent tutorial~\cite{JH-AI-LM-MM-ES-WW:18}. To the best of our knowledge, this 
is the first work in the literature that establishes a connection between the problem 
of time-varying optimization and  that of output regulation.

\textit{Contributions.} This paper features four main contributions. First, we 
recast the problem of designing a time-varying optimization algorithm as a 
nonlinear multivariable regulation problem, whereby the signal to be regulated 
is the gradient of the loss function. We leverage this formulation to derive 
necessary and sufficient conditions to achieve exact asymptotic tracking for 
a large class of optimization methods (described by sufficiently-smooth 
functions).
In a net departure from existing approaches (e.g., 
\cite{MZ:03,AS-ED-SP-GL-GG:20,MF-SP-VP-AR:17,bastianello2024}), our 
characterization is general and allows us to study not only a single 
optimization method, but an entire class, which enables us to derive 
fundamental results for all algorithms in this class. 
Second, by harnessing tools from center manifold theory~\cite{JC:81,JH-WW:84},
we provide necessary and sufficient conditions for an optimization algorithm 
to ensure tracking. Interestingly, these conditions depend 
on the properties of the loss function (through a gradient invertibility-type 
condition) and on the inner model describing the temporal variability of the 
problem. This property allows us to prove the \textit{internal model 
principle of time-varying optimization,} which states that for an optimization 
algorithm to achieve asymptotic tracking, it must incorporate a reduplicated 
model of the temporal variability of the problem. This feature is implicit in 
all existing approaches for time-varying 
optimization~\cite{YZ-MS:98,MF-SP-VP-AR:17,RR-AM-UV:22,RR-AM-UV:23,KG-DP:20} 
but, to the best of the authors' knowledge, lacked a rigorous understanding 
until now.  
Third, we derive necessary and sufficient 
conditions for the existence of a tracking algorithm. Fourth, we use our 
characterizations to design algorithms for time-varying optimization. 
With respect to the existing literature, our algorithm does not require one to 
know or measure exactly the temporal variability of the problem and thus uses 
less stringent assumptions. Finally, we illustrate the applicability of the 
approach numerically on both synthetic problems and the traffic assignment problem in transportation.

\textit{Organization.} 
Section~\ref{sec:prob_formulation} presents the problem of interest, 
Section~\ref{sec:parameter_feedback} the parameter-feedback problem (where 
the temporal variability of the problem can be explicitly measured), and 
Section~\ref{sec:gradient_feedback} the gradient-feedback problem (where the 
algorithm has access only to first-order function evaluations). 
Section~\ref{sec:fidelity_model} discusses the tracking accuracy in 
relationship to the fidelity of the internal model, 
Section~\ref{sec:extensions} presents extensions to constrained 
optimization and discrete-time algorithms,
Section~\ref{sec:simulations} validates numerically the results,  and 
Section~\ref{sec:conclusions} illustrates our conclusions. Finally, in the 
Appendix, we summarize basic facts on center manifold theory used in the paper.

\textit{Notation.}
We denote by $\symmetric^n$ the space of $n\times n$ symmetric real matrices. 
Given an open set $U,$ we say that $\map{f}{U}{\real}$ is of 
differentiability class $C^k$ if it has a $k\textsuperscript{th}$ derivative 
that is continuous in $U.$ 
%
The gradient of $f(x, \theta): \real^n \times \Theta \rightarrow \real, \Theta \subseteq \real^p,$ with respect to ${x} \in \real^n$ is denoted by 
$\nabla_{{x}} f({x}, \theta): \real^n \times \Theta \rightarrow \real^n$. The 
partial derivatives of 
$\nabla_{{x}} f({x}, \theta )$ with respect to ${x}$ and $\theta$ are denoted by $\nabla_{{x x}} f({x}, \theta): \real^n \times \Theta \rightarrow \mathbb{S}^n$ and $\nabla_{{x} \theta} f({x}, \theta): \real^n \times \Theta \rightarrow \real^{n\times p}$, respectively.

\section{Problem setting}
\label{sec:prob_formulation}

\subsection{Optimization objectives}
\label{sec:prob_formulation_a}
We consider the time-varying optimization problem:
\begin{align}\label{eq:optimization_objective_f}
    \min _{x \in \real^n} f(x,\theta(t)),
\end{align}
where $t \in \realnneg$ denotes time and 
$f: \real^n \times \Theta \rightarrow \real,$ 
$\Theta \subseteq \real^p,$ is a loss function that is parametrized by the 
time-varying parameter vector $\map{\theta}{\realnneg}{\Theta}.$
We make the following standard assumptions on the loss. 

\begin{assumption}[\textbf{\textit{Properties of the objective function}}]
\label{as:convexity_lipschitz_f0}
The map $x \mapsto f(x,\theta)$ is convex and $x\mapsto \nabla_x f(x,\theta)$ is Lipschitz 
continuous in $\real^n$, for each $\theta \in \Theta.$
\QEDB\end{assumption}

Convexity and smoothness are standard assumptions in 
optimization~\cite{EH:16}, which have been widely used in works on related problems~\cite{AC-EM-JC:16,AJ-AR-SS-KS:15,GB-JC-JP-ED:21-tcns,AS-ED-JM-AB:21,YZ-MS:98,MF-SP-VP-AR:17,bastianello2024,AS-ED-SP-GL-GG:20,MZ:03,AH-ZH-SB-GH-FD:24,AS-ED:17,OD-NH-PJ:17,MM-JB-JS-PT:24}.

In~\eqref{eq:optimization_objective_f}, the parameter $\theta(t)$ is used to 
model the temporal variability of the problem.
We will require that $\theta(t)$ belongs to a certain class of temporal 
variabilities, as specified next. 

\begin{assumption}[\textbf{\textit{Class of temporal variabilities}}]
\label{as:exosystem}
There exists a smooth (i.e., $C^\infty$) vector field 
$\map{s}{\Theta}{\real^p}$ and $\theta(0) \in \Theta$ such that the parameter vector 
$\theta(t)$ satisfies
\begin{align}\label{eq:exosystem}
\dot \theta (t) = s(\theta(t)),
\end{align}
for $t\in \realnneg.$
Moreover, $\theta=0$ is an equilibrium of~\eqref{eq:exosystem} and the  
trajectories of~\eqref{eq:exosystem} are bounded.~
\QEDB\end{assumption}

We stress that, a priori, we do not assume that $s(\theta)$ nor $\theta(0)$
are known (see Problem~\ref{prob:assumptions} shortly 
below for additional details). 
Assumption~\ref{as:exosystem} characterizes the class of temporal variabilities 
taken into consideration. This assumption is mild, 
as it only requires that $\theta(t)$ is deterministic, sufficiently smooth (so 
that its derivative is some $C^\infty$ function of $\theta(t)$ as 
in~\eqref{eq:exosystem}), and its trajectories remain bounded. In line with~\cite{AI-CIB:90,ED:76}, we call the autonomous 
system~\eqref{eq:exosystem} the {\it exosystem}.

For simplicity of the presentation, we assume that $\Theta$ is some neighborhood of the origin of $\real^p$, \blue{corresponding to the requirement that $\theta=0$ is an equilibrium of~\eqref{eq:exosystem}.}
We put no restrictions on 
the size of this neighborhood (which is, e.g., allowed to be the entire space 
$\Theta = \real^p$), and thus on the size of $\theta(t)$ nor on its temporal 
variation. Moreover, there is no 
restriction with asking that \blue{$\theta=0$ is an equilibrium because, if the equilibrium is located at any other point, the latter can be shifted to the origin through a time-invariant change of variables that does not alter the critical points of~\eqref{eq:optimization_objective_f} or the validity of Assumption~\ref{as:convexity_lipschitz_f0}.}

\begin{remark}[\textbf{\textit{Parametrization in time-varying optimization}}]
\label{rem:f0_x_t}
In~\eqref{eq:optimization_objective_f}, the time variation is captured 
implicitly by the parameter $\theta(t)$. A related, yet slightly 
more general, problem is:
\begin{align}\label{eq:optimization_objective_f0}
 \min_{x \in \real^n} f_0(x,t),
\end{align}
where the dependency on time is explicit. 
Problem~\eqref{eq:optimization_objective_f} can be recast uniquely
as in~\eqref{eq:optimization_objective_f0} by letting 
$f_0(x,t) = f(x,\theta(t))$ for all $x$ and $t$. On the other hand, in general, 
there exists an infinite number of ways to 
parametrize~\eqref{eq:optimization_objective_f0} as 
in~\eqref{eq:optimization_objective_f}, thus leading to possible 
ambiguities. For instance, any $f_0(x,t)$ may be parametrized by 
$\theta(t) = t$ (so that $f_0\equiv f$), although this is not
compatible with our boundedness trajectory assumption  
(Assumption~\ref{as:exosystem}).
\QEDB\end{remark}

\begin{remark}[\textbf{\textit{Discrete-time implementations}}]
In this work, we pursue the continuous-time
time-varying optimization problem~\eqref{eq:optimization_objective_f}, together with a 
continuous-time optimization algorithm. This is inspired by control applications 
(see~\cite{GB-JC-JP-ED:21-tcns} and references therein), 
where~\eqref{eq:optimization_objective_f} models 
performance objectives associated with a physical plant to be controlled, and 
$\theta(t)$ models exogenous disturbances affecting the plant. If instead the parameter vector were to evolve in discrete time, we would have the alternative optimization problem
\begin{align}\label{eq:discrete_optimization}
\min _{x \in \mathbb{R}^n} f\left(x, \theta_k\right),
\end{align}
where $k \in \mathbb{N}_{\geq 0}$ denotes time or iteration, and 
$f: \mathbb{R}^n \times \Theta \rightarrow \mathbb{R}$. While we could discretize $\theta(t)$ and then study discrete-time algorithms to solve this problem, physical control applications motivate studying the continuous-time problem directly.
We provide additional details on discrete-time implementations of our method in Section~\ref{sec:discrete-time}.
\QEDB\end{remark}

In what follows, we say that $x^\star(t),$ with 
$\map{x^\star}{\realnneg}{\real^n}$, is a \textit{critical trajectory} 
of~\eqref{eq:optimization_objective_f} if it satisfies:
\begin{align}\label{eq:critical_trajectory}
0=\nabla_x f(x^\star(t),\theta(t)), \quad \forall t \in \realnneg.
\end{align}
We assume the existence of a critical trajectory and that any critical 
trajectory is continuous. 
The existence of a critical trajectory can be guaranteed under standard 
assumptions on the optimization problem; for example, 
coercivity of the cost~\cite{RS:96}
(i.e., $f(x,\theta)\to \infty$ when $\Vert{x}\Vert \to \infty$), or (by 
Weierstrass' theorem~\cite{RS:96}) when the search domain can be restricted to 
a compact set without altering the optimizers.
Continuity of the critical trajectories can also be ensured under standard 
assumptions: for example, by Berge's theorem~\cite{RS:96}, by requiring that 
$f(x,\theta)$ is continuous in $\theta$.

In our analysis, we will use a {\it critical point with 
$\theta$ at rest}, which is a constant vector $x^\star_\circ \in \real^n$, 
defined implicitly as:
\begin{align}\label{eq:gradient_circ}
0=\nabla_x f(x^\star_\circ,0).
\end{align}
We assume that $x^\star_\circ$ exists and is locally unique\footnote{Existence 
and local uniqueness of $x^\star_\circ $ can be guaranteed, for 
instance, under the assumptions of the Implicit Function Theorem~\cite{WR:76}. 
Namely, let $X_{\circ} \times \Theta_{\circ}$ be some neighborhood of 
$\left(x_{\mathrm{o}}^{\star}, 0\right)$, $x^\star_\circ $  exists and is locally unique when:
(i) $f$ is $C^1$ on $X_{\circ} \times \Theta_{\circ}$,
(ii) $x \mapsto f(x, \theta)$ is $C^2$ on $X_{\circ}$ for each $\theta \in \Theta_{\circ}$, and
(iii) the $\left.\nabla_{x x}^2 f(x, \theta)\right|_{x=x_{\circ}^{\star}, \theta=0}$ is locally positive definite.}.

\begin{remark}[\textbf{\textit{Critical trajectory vs. critical point at rest}}]
\label{rem_ucirc_ustar}
We stress that $x^\star(t)$ and $x^\star_\circ$ are distinct quantities. 
Unlike the critical trajectory (defined in~\eqref{eq:critical_trajectory}), 
which denotes the optimal input to be tracked by the optimization algorithm, 
the critical point with $\theta$ at rest (defined 
in~\eqref{eq:gradient_circ}) is a constant vector that need not be tracked and 
may never coincide with $x^\star(t)$.~
\QEDB\end{remark}

We conclude this discussion by reviewing two important classes of optimization 
methods commonly used to solve~\eqref{eq:optimization_objective_f}.

\begin{remark}[\textbf{\textit{Basic gradient flow algorithms}}]
\label{rem:static_gradient}
The basic gradient flow algorithm~\cite{PAA-KK:06,EH:16} 
for~\eqref{eq:optimization_objective_f} reads as:
\begin{align}\label{eq:gradient_flow}
    \dot x(t) = -\eta \nabla_x f(x(t), \theta(t)),
\end{align}
where $\eta>0$ models the algorithm's step size. It is 
known~\cite{PAA-KK:06} that this class of algorithms is capable of 
computing (asymptotically) critical points of~\eqref{eq:optimization_objective_f} \blue{only when the latter are constant;} in all other cases, the algorithm converges to a neighborhood of a critical point and convergence is inexact.
We refer the reader to Remark~\ref{rem:internal_model_grad_flow} for an 
insightful interpretation of this fundamental limitation. 
\QEDB\end{remark}

\begin{remark}[\textbf{\textit{Prediction-correction algorithms}}]
\label{rem:ppred_correct}
An established family of methods to solve~\eqref{eq:optimization_objective_f} 
is that of prediction-correction 
algorithms~\cite{YZ-MS:98,MF-SP-VP-AR:17,RR-AM-UV:22,RR-AM-UV:23,KG-DP:20}. 
Tailored to our setting, the basic method of this class reads as:
\begin{multline}\label{eq:pred_correction}
    \dot{x}(t) = -\nabla_{xx}^\inv f(x(t), \theta(t)) [\eta \nabla_x f(x(t), \theta(t))\\
    +\nabla_{x \theta} f(x(t), \theta(t)) \cdot s(\theta(t))],
\end{multline}
where $\nabla_{x x} f(x , \theta)$ is the Hessian of $f(x, \theta)$, and
$\nabla_{x \theta} f(x, \theta)$ is the partial derivative of 
$\nabla_x f(x, \theta)$ with respect to $\theta$.
Under strong convexity assumptions, this algorithm is known to converge 
exponentially, with zero asymptotic error, to a critical trajectory~\cite{YZ-MS:98,MF-SP-VP-AR:17}.
Variations of this algorithm have also been proposed to achieve fixed-time or 
finite-time convergence~\cite{RR-AM-UV:22}.
We conclude by noting that implementing~\eqref{eq:pred_correction} requires (i) 
the cost to be twice continuously differentiable and strongly convex, (ii) knowledge of the maps $s(\theta)$, $\nabla_x f(x,\theta)$, $\nabla_{xx} f(x,\theta)$, and $\nabla_{x\theta} f(x,\theta)$, and (iii) knowledge of $\theta(t).$
\QEDB\end{remark}

\subsection{Problem statement}

In line with the literature~\cite{EH:16,YZ-MS:98,MF-SP-VP-AR:17,RR-AM-UV:22,bastianello2024}, we 
focus on gradient-type algorithms for 
solving~\eqref{eq:optimization_objective_f}, which are algorithms that have 
access to first-order oracles of the cost; that is, function evaluations of the 
map:
\begin{align}\label{eq:gradient}
  (t, x) \mapsto \nabla_x f(x,\theta(t)).
\end{align}

\begin{remark}[\textbf{\textit{Evaluating the gradient}}]
\label{rem:evaluating_gradient}
In the optimization literature~\cite{EH:16}, function evaluations 
of~\eqref{eq:gradient} are generally obtained through two main approaches:
(i) the algorithm has access to both the analytical form of 
$\nabla_x f(x,\theta)$ and $\theta(t)$ at each time $t$ (either because this 
signal is known or measurable), or (ii) an oracle returns evaluations 
of~\eqref{eq:gradient} (computed, e.g., through numerical differentiation). 
These alternatives will be discussed more in detail in (O1)--(O4), shortly below.
\QEDB\end{remark}

We let the optimization algorithm be described by a dynamic internal state 
$z(t),$ taking values on an open subset 
$\mc Z  \subseteq \real^{n_c}, n_c \in \naturalpos$. 
The optimization algorithm generates a sequence of points $x(t) \in \real^{n}$ 
(called \textit{exploration signal}), and has access to function 
evaluations of~\eqref{eq:gradient} at these points (called \textit{gradient 
feedback signal}):
\begin{subequations}\label{eq:controller_equatons_zx_f0}
\begin{equation}\label{eq:controller_equatons_zx_f0_b}
    y(t) = \nabla_x f(x(t),\theta(t)).
\end{equation}
Together with this gradient feedback signal, the optimization algorithm is described by\footnote{While one could 
allow $G_c$ to also depend on $y$, so that $x(t) = G_c(z(t),y(t))$, we will 
show in Section~\ref{sec:gradient_feedback} that this is unnecessary.}
\begin{align}\label{eq:controller_equatons_zx_f0_a}
\dot z(t) &= F_c(z(t), y(t)), & 
x(t) &= G_c(z(t)), 
\end{align}
\end{subequations}
where $\map{F_c}{\mc Z \times \real^n}{\real^{n_c}}$ and 
$\map{G_c}{\mc Z}{\real^{n}}$ are functions to be designed.
In the remainder, we refer to~\eqref{eq:controller_equatons_zx_f0} as a 
\textit{dynamic gradient-feedback optimization algorithm}.
The architecture of the presented gradient-feedback algorithm 
is illustrated in Fig.~\ref{fig:feedback_design}.
For simplicity of the presentation, we will require that $F_c(z,y)$ and 
$G_c(z)$ are such that\footnote{Notice that this is without loss 
of generality since  $F_c(z,y)$ and $G_c(z)$ are to be designed 
and $x_\circ^\star$ is known through~\eqref{eq:gradient_circ}.}:
\begin{align}\label{eq:conditions_Fc_Gc}
0 &= F_c(z_\circ^\star, 0) \qquad\text{and}\qquad
x_\circ^\star = G_c(z_\circ^\star),
\end{align}
for some locally unique $z_\circ^\star \in \real^{n_c}$.
Together with~\eqref{eq:gradient_circ}, this ensures that $z_\circ^\star$ is an equilibrium 
of~\eqref{eq:controller_equatons_zx_f0} when the gradient 
feedback signal $y(t)$ is equal to zero.

The dynamics of 
algorithm~\eqref{eq:controller_equatons_zx_f0}, coupled with the 
exosystem~\eqref{eq:exosystem}, have the form of a nonlinear 
autonomous system:
\begin{subequations}\label{eq:copled_system_gradient_feedback}
    \begin{align}
    \dot z(t) &= F_c(z(t), y(t)),\\
    y(t) &= \nabla_x f(G_c(z(t)),\theta(t)),\\
    \dot \theta (t) &= s(\theta(t)). 
    \end{align}
\end{subequations}

\begin{definition}[\textbf{\textit{Exact asymptotic tracking}}]
\label{defn:track-critical-trajectory}
We say that~\eqref{eq:controller_equatons_zx_f0} 
\textit{exactly asymptotically tracks a critical trajectory} 
of~\eqref{eq:optimization_objective_f} if there exists a neighborhood of $(z_\circ^*,0)$ in $\mathcal{Z}\times\Theta$ such that, for each initial 
condition $(z(0),\theta(0))$ in the neighborhood, the solution 
of~\eqref{eq:copled_system_gradient_feedback} satisfies $y(t) \to 0$ as 
$t\to\infty$. 
\QEDB\end{definition}

By considering the entire class of optimization algorithms as 
in~\eqref{eq:controller_equatons_zx_f0}, and without requiring any assumption 
in addition to~\eqref{eq:gradient}, we begin by posing the following 
fundamental question.

\setcounter{problem}{-1}
\begin{problem}[\textbf{\textit{Minimal knowledge for exact asymptotic tracking}}]
\label{prob:assumptions}
Consider the class of optimization 
algorithms~\eqref{eq:controller_equatons_zx_f0}.
Determine the minimal necessary knowledge (beyond~\eqref{eq:gradient}), 
concerning the exosystem~\eqref{eq:exosystem} and the 
optimization~\eqref{eq:optimization_objective_f}, needed to 
design an algorithm from this class that tracks, with zero asymptotic error, 
the critical trajectories of~\eqref{eq:optimization_objective_f}.
\QEDB\end{problem}


Before addressing this question, we first outline various pieces of information on the exosystem and objective that the algorithm may (or may not) be able to access.
Two (non mutually-exclusive\footnote{Notice that requiring both (E1) and (E2) is equivalent to asking that both $s(\theta)$ as well as $\theta(0)$ are known.}) 
assumptions on the exosystem from the literature
(see, e.g., \cite{PAA-KK:06,EH:16,YZ-MS:98,MF-SP-VP-AR:17,RR-AM-UV:22,bastianello2024}) are:
\begin{enumerate}
\item[(E1)] The parameter signal $\theta(t)$ is known or measurable at each $t \in \realnneg,$ but no knowledge of $s(\theta)$ is available. 
\item[(E2)] The vector field $s(\theta)$ is known, but $\theta(0)$ (and hence 
also $\theta(t)$) is unknown.
\end{enumerate}
A second distinction can be made in terms of what knowledge concerning the 
objective function is available. In this 
case, several (non mutually-exclusive) options can be found in the 
literature (see, e.g.,~\cite{PAA-KK:06,EH:16,YZ-MS:98,MF-SP-VP-AR:17,RR-AM-UV:22,bastianello2024}):
\begin{enumerate}
\item[(O1)] The algorithm has access to oracle evaluations of~\eqref{eq:gradient}.
\item[(O2)] The analytical form of the gradient
$\nabla_{x} f(x,\theta)$ is known. 
\item[(O3)] The sensitivity $\nabla_{x\theta} f(x,\theta)$ of the cost with respect to the parameter is known. 
\item[(O4)] The Hessian $\nabla_{xx}^2 f(x,\theta)$ of the cost is known.
\item[(O5)] The objective $f(x,\theta)$ is quadratic in both $x$ and $\theta$.
\end{enumerate}

\begin{table*}[t]
\caption{Summary of required knowledge and known guarantees by various 
established approaches. The symbol $(\checkmark)$ is used when (O1) follows 
from (E1) and (O2).  See Remarks
\ref{rem:static_gradient}, \ref{rem:ppred_correct}, and 
\ref{rem:special_cases_optimization_alg} for discussions. \blue{Notice that particular assumptions on the optimization~\eqref{eq:optimization_objective_f} and the exosystem~\eqref{eq:exosystem} may differ across works.}}
\setlength{\tabcolsep}{4pt}
\centering\begin{tabular}{l|c|cc|ccccc}
  & Approximate & (E1) & (E2) & (O1) & (O2) & (O3) & (O4) & (O5) \\
    & solutions & $\theta(t)$ & $s(\theta)$ & 
    $(t,x)\mapsto \nabla_x f(x,\theta(t))$ & 
    $\nabla_x f(x, \theta)$ & 
    $\nabla_{x\theta} f(x, \theta)$& $\nabla_{xx}^2 f(x, \theta)$ & $f$ quadratic\\
  \hline
  Basic gradient flow~\cite{PAA-KK:06,EH:16} & \checkmark & & & \checkmark & & & & \\
  \rowcolor{gray!25}Control-based methods~\cite{bastianello2024} & &  & \checkmark & \checkmark &\checkmark  & \checkmark  & & \checkmark \\
  Prediction-correction~\cite{YZ-MS:98,MF-SP-VP-AR:17,RR-AM-UV:22} & & \checkmark & \checkmark & (\checkmark) & \checkmark & \checkmark & \checkmark & \\
  \rowcolor{gray!25} This work (Section~\ref{sec:parameter_feedback}) & & \checkmark & & (\checkmark)  & \checkmark & & &  \\
  This work (Section~\ref{sec:gradient_feedback}) & & & \checkmark & \checkmark & \checkmark &  &  & 
\end{tabular}
\label{tab:comparison_assumptions}
\end{table*}

These are not assumptions we impose a priori in this work, but rather widely used requirements in the 
literature~\cite{PAA-KK:06,EH:16,YZ-MS:98,MF-SP-VP-AR:17,RR-AM-UV:22,bastianello2024}.
Various methods have been developed in the literature based on 
different combinations of these assumptions: for example, under (E1) and (O1), 
the basic gradient flow algorithm~\eqref{eq:gradient_flow} converges 
(asymptotically) to a neighborhood of the critical points 
(see Remark~\ref{rem:static_gradient}). 
Under assumptions (E1), (E2), (O2), (O3), and (O4), the prediction-correction 
method~\eqref{eq:pred_correction} 
converges (asymptotically or in fixed-time) exactly to the optimizers
(see Remark~\ref{rem:ppred_correct}). 
Table~\ref{tab:comparison_assumptions} provides a comparison of the 
assumptions required by established methods and their positioning with respect to this work.

\begin{remark}[\textbf{\textit{Generality of the algorithm class}}]
\label{rem:special_cases_optimization_alg}
The gradient flow algorithm~\eqref{eq:gradient_flow} can be viewed  as a 
special case of~\eqref{eq:controller_equatons_zx_f0} with $F_c(z,y) = -\eta y$ and $G_c(z,y) = z$.
Similarly, the prediction-correction method~\eqref{eq:pred_correction} can 
be viewed as a special case of~\eqref{eq:controller_equatons_zx_f0} with
$z=(z_1, z_2)$ and 
\begin{align}\label{eq:pred_corr_FcGc}
F_c(z,y) &=\begin{bmatrix}
-\nabla_{x x}^\inv f(z_1 , z_2)  [y +\nabla_{x \theta} f(z_1, z_2) \cdot s(z_2)]\\
s(z_2)
\end{bmatrix}, \nonumber\\
G_c(z) &= z_1,
\end{align}
with $z_2(0) = \theta(0).$
It follows that~\eqref{eq:controller_equatons_zx_f0} is general enough to 
encompass these two established methods as special cases, and thus all 
conclusions concerning the class of 
algorithms~\eqref{eq:controller_equatons_zx_f0} made here will also apply to 
these methods.~
\QEDB\end{remark}

Our second objective is to design optimization algorithms from the 
class~\eqref{eq:controller_equatons_zx_f0} that track, with zero asymptotic 
error, the critical trajectories of~\eqref{eq:optimization_objective_f}. 
We formalize this notion next. 




\begin{problem}[\textbf{\textit{Design of dynamic gradient-feedback optimization algorithms}}]
\label{prob:grad_feedback}
Design $F_c(z,y),$  $G_c(z),$ and $n_c,$ requiring the minimal necessary 
knowledge for exact asymptotic tracking (as in Problem~\ref{prob:assumptions}), 
so that~\eqref{eq:controller_equatons_zx_f0} exactly asymptotically tracks a 
critical trajectory of~\eqref{eq:optimization_objective_f}.~
\QEDB\end{problem}

\section{The parameter-feedback problem}
\label{sec:parameter_feedback}

In many cases of interest~\cite{EH:16}, having access to 
$y(t)$ is a byproduct of having access to both the function
$\nabla_x f(x,\theta)$ and knowledge or measurements of 
$\theta(t)$; see Remark~\ref{rem:evaluating_gradient}. 
In these cases, assumptions (E1) and (O2) are automatically satisfied.
In this section, we analyze this scenario. 
We anticipate that the results derived in this section will also serve as an 
intermediate step to tackle the more challenging problem  where assumption (E1) 
is relaxed, which is the focus of Section~\ref{sec:gradient_feedback}. 

When the algorithm has access to both $\theta(t)$ at each time 
$t \in \realnneg,$ and the function $\nabla_x f(x,\theta)$, the measurements 
$y(t)$ do not provide additional information.
Hence, under the assumptions of this section, the optimization 
algorithm~\eqref{eq:controller_equatons_zx_f0} can be replaced by 
the algebraic relationship\footnote{While one could consider a dynamic 
optimization algorithm of the form $\dot z(t) = F_c(z(t), \theta(t))$ and 
$x(t) = G_c(z(t)),$ we will prove in 
Theorem~\ref{thm:solvability_parameter_feedback} that a dynamic structure is unnecessary.}
\begin{align}\label{eq:parameter_feedback}
x(t) = H_c(\theta(t)),
\end{align}
where $\map{H_c}{\Theta}{\real^n}$ is a mapping to be designed.
Because of the explicit dependence on $\theta(t),$ we will refer 
to~\eqref{eq:parameter_feedback} to as a \textit{parameter-feedback}  optimization algorithm. 
In analogy with~\eqref{eq:conditions_Fc_Gc}, for simplicity 
of the presentation, in this section, we will impose that 
$H_c(\theta)$ satisfies:
\begin{align*}
x^\star_\circ = H_c(0),
\end{align*}
and require that $H_c$ is of class $C^0.$

The composition of~\eqref{eq:exosystem}, 
\eqref{eq:controller_equatons_zx_f0_b}, and \eqref{eq:parameter_feedback} is given by:
\begin{subequations}\label{eq:copled_system_parameter_feedback}
    \begin{align}
    y(t) &= \nabla_x f(H_c(\theta(t)),\theta(t)), \\
    \dot \theta (t) &= s(\theta(t)).
    \end{align}
\end{subequations}
In analogy with Definition~\ref{defn:track-critical-trajectory}, 
we will say that the parameter-feedback algorithm~\eqref{eq:parameter_feedback} 
\textit{exactly asymptotically tracks a critical trajectory} 
of~\eqref{eq:optimization_objective_f} if there exists a neighborhood $\Theta_s \subset \Theta$ of the origin such that, for each initial 
condition $\theta(0) \in \Theta_s$, the solution 
of~\eqref{eq:copled_system_parameter_feedback} satisfies $y(t) \to 0$ as 
$t\to\infty$.

For the algorithm~\eqref{eq:parameter_feedback} considered in this section,
Problems~\ref{prob:assumptions} and \ref{prob:grad_feedback} are reformulated, respectively, 
as follows. 

\begin{problem}[\textbf{\textit{Minimal knowledge for parameter feedback}}]
\label{prob:minimal_knowledge_parameter}
Consider the class of optimization algorithms~\eqref{eq:parameter_feedback}.
Determine the minimal necessary knowledge 
concerning~\eqref{eq:optimization_objective_f} and~\eqref{eq:exosystem}, needed 
to design an algorithm from this class that exactly asymptotically tracks a 
critical trajectory of~\eqref{eq:optimization_objective_f}.
\QEDB\end{problem}

\begin{problem}[\textbf{\textit{Design of parameter-feedback optimization algorithms}}]
\label{prob:parameter_feedback}
Design $H_c(\theta),$  requiring minimal knowledge for parameter-feedback (as 
in Problem~\ref{prob:minimal_knowledge_parameter}), 
so that~\eqref{eq:parameter_feedback} exactly asymptotically tracks a 
critical trajectory of~\eqref{eq:optimization_objective_f}.~
\QEDB\end{problem}



\subsection{Fundamental results}
Solvability of the parameter-feedback problem will depend on the existence of a function that zeros the gradient on the set of limit points of the exosystem; we now define these notions.

\begin{definition}[\textbf{\textit{Mapping zeroing the gradient}}]\label{defn:gradient-invertibility}
We say that a mapping $H_c : \Theta\to\real^n$ \textit{zeros the gradient} at 
the point $\theta\in\Theta$~if
\begin{align}\label{eq:gradient_condition}
    0 &= \nabla_x f(H_c(\theta),\theta).
\end{align}
Moreover, we say that $H_c$ zeros the gradient
\textit{on a set} $\Theta_\circ \subseteq \Theta$ if~\eqref{eq:gradient_condition} holds for all $\theta \in \Theta_\circ$.
\QEDB\end{definition}



\begin{definition}[\textbf{\textit{Limit point and limit set}}]
A point $\theta_\omega \in \Theta$ is a \emph{limit point with respect 
to the initialization $\theta_\circ \in \Theta$}
if there exists a sequence $\{t_i\}_{i\in\naturalnneg}$ with 
$t_i\to\infty$ as ${i\to\infty}$ such that the exosystem~\eqref{eq:exosystem}
with $\theta(0) = \theta_\circ$ satisfies $\theta({t_i})\to\theta_\omega$ as 
${i\to\infty}$. 
Let $\Omega(\theta_\circ)$ denote the set of all limit points 
(i.e., for all sequences $t_i$) 
of~\eqref{eq:exosystem} with respect to the initialization
$\theta_\circ\in\Theta.$
Given $\Theta_\circ \subseteq \Theta,$ the set
$\Omega(\Theta_\circ) := 
\cup_{\theta_\circ \in \Theta_\circ} \Omega(\theta_\circ)
$ 
is called the \textit{limit set with respect 
to initializations in $\Theta_\circ$}~\cite{birkhoff}.~
\QEDB\end{definition}

Intuitively, $\Omega(\Theta_\circ)$ denotes the set of all limit 
points (equilibria, limit cycles, etc.) that can be reached by the exosystem 
when initialized at points in $\Theta_\circ$. 
By the boundedness trajectories assumption 
(Assumption~\ref{as:exosystem}), $\Omega(\Theta_\circ)$ is a bounded set.


The following result characterizes all parameter-feedback optimization
algorithms that achieve asymptotic tracking.

\begin{theorem}[\textbf{\textit{Existence and characterization of 
parameter-feedback algorithms}}]
\label{thm:solvability_parameter_feedback}
Let Assumptions~\ref{as:convexity_lipschitz_f0}--\ref{as:exosystem} hold. The 
parameter-feedback algorithm~\eqref{eq:parameter_feedback} 
asymptotically tracks a critical trajectory 
of~\eqref{eq:optimization_objective_f} if and only if there exists a 
neighborhood $\Theta_\circ\subset \Theta$ of the origin such that the mapping 
$H_c$ zeros the gradient on the limit set $\Omega(\Theta_\circ)$.
\end{theorem}

\begin{proof}
\textit{(Only if)} Suppose $y(t)\to 0$ as $t\to\infty$ for initializations 
$\theta(0) \in \Theta_\circ$; 
we will show that $H_c$ zeros the gradient on $\Omega(\Theta_\circ)$. 
By Assumption~\ref{as:exosystem}, the trajectories of the exosystem are bounded,
and thus, by the Bolzano–Weierstrass theorem, there exists an increasing 
subsequence $\{t_i\}_{i\in\naturalnneg}$ such that $\theta(t_i)$ converges to 
some limit point $\theta_\omega \in \Omega\left(\theta(0)\right)$. We then have:
\begin{align}\label{eq:limit_y_sequence}
\lim_{i\to\infty} y({t_i})
    = \lim_{i\to\infty} \nabla_x f(H_c(\theta({t_i})),\theta({t_i}))
    = \nabla_x f(H_c(\theta_\omega),\theta_\omega),
\end{align}
where the second identity follows by the continuity of the gradient (see 
Assumption~\ref{as:convexity_lipschitz_f0}) and that of $H_c$.
Because, $y(t)\to 0$, the left-hand side of~\eqref{eq:limit_y_sequence} is 
equal to zero, and thus~\eqref{eq:limit_y_sequence} implies that
$\nabla_x f(H_c(\theta_\omega),\theta_\omega)=0.$ 
Since this holds for any limit point $\theta_\omega\in \Omega(\Theta_\circ)$, 
the statement follows.

\textit{(If)} Suppose $\theta(0) \in \Theta_\circ$ and that $H_c$ zeros the 
gradient on $\Omega(\Theta_\circ)$. 
Then, the right-hand side of~\eqref{eq:limit_y_sequence} is equal to zero, 
which implies the existence of a sequence $t_i$ such that $y({t_i})\to 0$ as 
${i\to\infty}$. Since this holds for any limit point 
$\theta_\omega\in\Omega(\theta(0))$, any convergent subsequence of~$y(t)$ 
converges to zero. Moreover,~$y(t)$ is bounded due to the Lipschitz continuity of 
the gradient, so~$y(t)$ also converges to zero as $t\to\infty$.
By iterating the reasoning for all $\theta(0) \in \Theta_\circ,$ it follows 
that $y(t) \to 0$ for all $\theta(0) \in \Theta_\circ$, and the claim
follows. 
\end{proof}

The theorem shows that the existence of a parameter-feedback algorithm
is dependent upon the existence of a mapping that zeros the gradient. An 
application of this theorem is illustrated in Example~\ref{ex:Hc_maynot_exist}, 
which shows that a parameter-feedback algorithm may fail to exist in some  
circumstances.
The theorem also provides a full characterization of all parameter-feedback 
algorithms that achieve asymptotic tracking. Loosely speaking, 
$x = H_c(\theta)$ is a parameter-feedback algorithm if and only if it zeros the 
gradient on the limit set of the exosystem.
As a byproduct, the result also provides a necessary and sufficient condition 
for the solvability of the parameter-feedback problem: the problem is 
solvable if and only if the set of solutions to the system of equations 
$0 = \nabla_x f(x, \theta)$ can be expressed, at all limit points of the 
exosystem, as the graph of a function $x = H_c(\theta).$
We present a set of sufficient conditions for this to hold in 
Section~\ref{sec:sufficient_conditions}.

By Theorem~\ref{thm:solvability_parameter_feedback}, the problem of designing 
a parameter-feedback algorithm for exact asymptotic tracking can be reduced to
that of finding a mapping $x = H_c(\theta)$ that zeros the gradient. Therefore, we have the following.

\textit{Answer to Problem~\ref{prob:minimal_knowledge_parameter}:}
When a parameter-feedback algorithm exists, the minimal knowledge needed to design such an 
algorithm is (O2). 
Notice that assumption (E1) is implicitly required for~\eqref{eq:parameter_feedback} to be implementable.

We illustrate the applicability of the result and the necessity of the provided 
condition in the following example.

\begin{figure}[t]
\centering \subfigure[]{\includegraphics[width=.48\columnwidth]{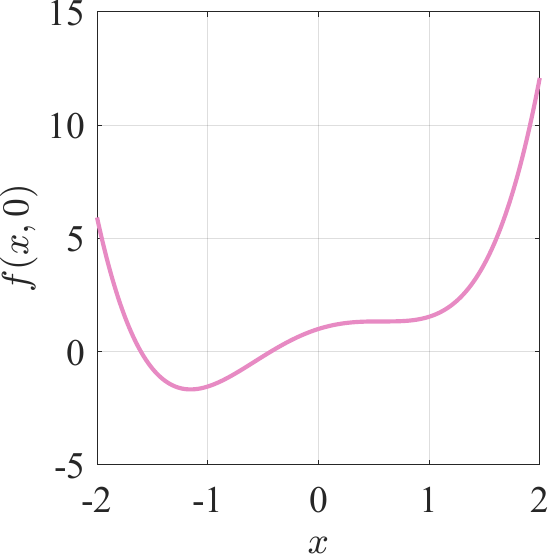}}
\hfill
\centering \subfigure[]{\includegraphics[width=.48\columnwidth]{
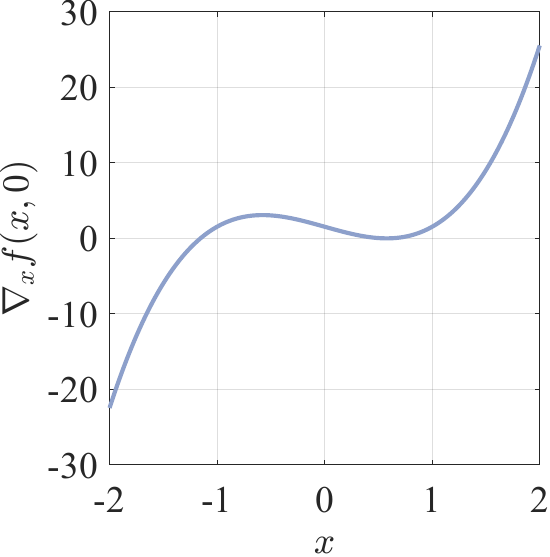}} 
\caption{Investigation of the condition~\eqref{eq:gradient_condition}. 
(Left) Loss function $f(x, \theta)$ studied in 
Example~\ref{ex:Hc_maynot_exist}, plotted for $\theta=0.$ 
(Right) Gradient $\nabla f(x, \theta).$
The function $f(x,0)$ admits two critical points: 
$x^\star_{\circ,1} = \frac{1}{\sqrt{3}}$ and 
$x^\star_{\circ,2} = -\frac{2}{\sqrt{3}}.$ At $x^\star_{\circ,1},$ condition
\eqref{eq:gradient_condition} is not satisfied, since with an upward shift 
of the graph of $\nabla_xf(x,0),$  $x^\star_{\circ,1}$ is no longer a critical 
point of $f(x,0).$
On the other hand, \eqref{eq:gradient_condition} holds for 
$x^\star_{\circ,2},$ since $x^\star_{\circ,2}$ varies continuously as $\theta$ 
is perturbed. 
See Example~\ref{ex:Hc_maynot_exist} for a discussion.
}
\vspace{-.5cm}
\label{fig:f_df}
\end{figure}

\begin{example}\label{ex:Hc_maynot_exist}
Consider an instance of~\eqref{eq:optimization_objective_f} with $n=p=1,$ 
\begin{align}
\label{eq:example_parameter_feedback}
    f(x, \theta) = (x-1)^2(x+1)^2 + \frac{8}{3\sqrt{3}}x + \theta x.
\end{align}
See Fig.~\ref{fig:f_df}(a) for an illustration of this function. For 
$\theta=0,$ the optimization problem associated 
with~\eqref{eq:example_parameter_feedback}  admits two critical points: 
$x^\star_{\circ,1} = \frac{1}{\sqrt{3}}$ and 
$x^\star_{\circ,2} = -\frac{2}{\sqrt{3}};$ indeed, it follows from direct inspection that
$\nabla_xf(x^\star_{\circ,1},0)=0$ and $\nabla_xf(x^\star_{\circ,2},0)=0.$
At the critical point $x^\star_{\circ,1},$ a function $H_c(\theta)$ as 
in~\eqref{eq:gradient_condition} does not exist. This can be visualized 
with the aid of Fig.~\ref{fig:f_df}(b): if $\theta=0$ is perturbed to 
$\theta+\epsilon,$ $\epsilon>0,$ the graph of $\nabla_xf(x,0)$
(illustrated in Fig.~\ref{fig:f_df}(b)) shifts upward and the equation 
$\nabla_xf(x,\epsilon)=0$ no longer admits a solution
in a neighborhood of $x^\star_{\circ,1}.$
%

On the other hand, for the critical point $x^\star_{\circ,2},$ any 
arbitrarily small upward or downward shift of the graph of 
$\nabla_xf(x,0)$ results in a continuous perturbation of  
$x^\star_{\circ,2}$ (see Fig.~\ref{fig:f_df}(b)),
thus suggesting existence of $x=H_c(\theta)$ as 
in~\eqref{eq:gradient_condition}. 
This graphical observation can be formalized with 
the aid of the implicit function theorem~\cite{WR:76}, as described next. 
Define  $F(x, \theta) := \nabla_x f(x,\theta)$ and notice that $F$ is 
continuously differentiable with  $F(x^\star_{\circ,2},0)=0.$ By the implicit 
function theorem, there exists a neighborhood $\Theta_\circ$ of 
$x^\star_{\circ,2}$ and a function $\map{H_c}{\Theta_\circ}{\real^n}$ 
such that $F(H_c(\theta),\theta)=0$ in $\Theta_\circ,$ provided that 
$\left. \frac{\partial F}{\partial x}\right|_{(x,\theta)=(x^\star_{\circ,2},0)} \neq 0.$ 
By inspection, it is immediate to see that the latter condition is satisfied. 
\QEDB\end{example}

We conclude by observing that the proof of 
Theorem~\ref{thm:solvability_parameter_feedback} is constructive in that 
it provides an explicit procedure to construct parameter-feedback 
algorithms. 

\textit{Answer to Problem~\ref{prob:parameter_feedback}:}
When a parameter-feedback algorithm exists, a procedure to design it  
requiring minimal knowledge is given by: 1) determine the limit set $\Omega(\Theta_\circ)$ for 
the given set of initial conditions $\Theta_\circ$; 2) find a mapping $H_c(\theta)$ that zeros 
the gradient on the limit set; and, 3) the parameter-feedback algorithm is given 
by~\eqref{eq:parameter_feedback}. 


We conclude by illustrating the parameter-feedback design procedure on a quadratic problem.

\begin{example}
\label{ex:quadratic_cost}
Consider an instance of~\eqref{eq:optimization_objective_f} with quadratic cost 
and time-variability that depends linearly on $\theta(t)$ (which, e.g., has 
been investigated in~\cite{bastianello2024}):
\begin{align}\label{eq:quadratic_cost_example}
    f(x(t),\theta(t)) = \tfrac{1}{2} x(t)^\tsp R x(t) + x(t)^\tsp Q \theta(t),
\end{align}
with matrices $R\in \symmetric^n$ and $Q \in \real^{n \times p}$. This loss is convex and Lipschitz smooth and thus satisfies 
Assumption~\ref{as:convexity_lipschitz_f0}.
In this case, the signal we wish to regulate to zero is: 
$$y(t) = \nabla_x f(x(t),\theta(t)) = R x(t) + Q \theta(t).$$ 
For arbitrary $\theta,$ this problem admits a critical point if and only if $\im Q \subseteq \im R$, in which case $x_\circ^\star$ is unique. Applying Theorem~\ref{thm:solvability_parameter_feedback} amounts to finding a linear transformation $H_c \in \real^{n \times p}$ such that $0 = (R H_c + Q) \theta$ for all $\theta$ in a neighborhood of the origin. Assuming $\im Q\subseteq \im R$, we can choose $H_c = - R^\dagger Q$, where $R^\dagger$ is the pseudo-inverse of $R$. Note that, by substituting into~\eqref{eq:copled_system_parameter_feedback}, 
we have ${y(t) = R H_c\theta(t) + Q \theta(t)=0}$ for all times $t \in \realnneg.$
Namely, the gradient is identically zero. This property holds because the mapping $H_c(\theta)$ derived here zeros the gradient globally (and not only in some neighborhood of the origin).
More generally, this feature will hold for any problem for which the gradient of the 
cost admits a mapping that zeros the gradient globally. A set of sufficient 
conditions to check for this property are given in 
Theorem~\ref{thm:suff_cond_parameter_feedback_global}.
\QEDB\end{example}

\begin{remark}[\textbf{\textit{Prediction-correction algorithms asymptotically 
compute maps that zero the gradient}}]
    As an illustration, consider the quadratic problem studied in Example~\ref{ex:quadratic_cost}.
    Let $x_\text{PC}(t)$ denote the iterates of the prediction-correction algorithm~\eqref{eq:pred_correction} applied to the quadratic objective~\eqref{eq:quadratic_cost_example}, and let $x_\text{PF}(t) = H_c \theta(t)$ with $H_c = -R^\dagger Q$ be the parameter-feedback algorithm. Then, the error $e(t) = x_\text{PC}(t) - x_\text{PF}(t)$ satisfies the dynamics $\dot e(t) = -e(t)$, indicating that the iterates of prediction-correction converge exponentially to those of the parameter-feedback algorithm, which zeros the gradient.
\QEDB\end{remark}

\subsection{Sufficient conditions for the existence of a parameter-feedback algorithm}
\label{sec:sufficient_conditions}

Existence of a map $x=H_c(\theta)$ as in~\eqref{eq:gradient_condition} can be 
ensured for general problems with the aid of the implicit function 
theorem~\cite{WR:76}, as illustrated by the following result.

\begin{proposition}[\textbf{\textit{Existence of local parameter-feedback algorithms}}]
\label{prop:suff_cond_parameter_feedback}
Let Assumptions~\ref{as:convexity_lipschitz_f0}--\ref{as:exosystem} hold, and let $X_\circ \times \Theta_\circ$ be some 
neighborhood of $(x_\circ^\star,0).$ 
Further, assume that: 
\begin{enumerate}
\item[(i)] the loss function $f$ is $C^1$ on $X_\circ \times \Theta_\circ,$
\item[(ii)] $x \mapsto f(x,\theta)$ is $C^2$ on $X_\circ$ for each $\theta \in \Theta_\circ,$ and
\item[(iii)] the Hessian
$\left. \nabla^2_{xx} f(x, \theta) \right|_{x=x^\star_\circ, \theta=0}$ is 
positive definite.
\end{enumerate}
Then, there exists a $C^0$ mapping $\map{H_c}{\Theta}{\real^n}$ that zeros the gradient on some neighborhood of the origin of $\Theta$.
\QEDB\end{proposition}

\begin{proof}
Define $F(x,\theta) = \nabla_x f(x,\theta)$ and note that, under the stated assumptions, $F$ is $C^0$ on $X_\circ \times \Theta_\circ$, the mapping $x \mapsto F(x,\theta)$ is $C^1$ on $X_\circ$ for all $ \theta \in \Theta_\circ$, and
\[
    \det \left[ \frac{\partial F(x,\theta)}{\partial x} \right]_{(x, \theta) = (x_\circ^\star,0)} \neq 0.
\]
Then, the result follows from the implicit function theorem (see, e.g., \cite[Thm.~1]{HN-FF:22}) applied to the first-order optimality conditions $0 = F(x, \theta)$ at the point $(x, \theta) = (x^\star_\circ,0)$.
\end{proof}

In other words, Proposition~\ref{prop:suff_cond_parameter_feedback} states that 
when the loss function is sufficiently smooth and the Hessian is locally 
positive definite, the existence of $H_c(\theta(t))$ is guaranteed by the 
implicit function theorem. Notice that these conditions are sufficient but not 
necessary.

\begin{remark}[\textbf{\textit{Computing mappings that zero the gradient}}]
Under the assumptions of 
Proposition~\ref{prop:suff_cond_parameter_feedback},
the implicit function theorem gives that the linear function
\begin{multline}\label{eq:Hc-linear}
\hat H_c(\theta) = x_\circ^\star -
\bigl( \nabla^2_{xx} f(x, \theta) \big|_{x=x^\star_\circ \theta=0}
\bigr)^\inv \\ \nabla_{x\theta} f(x, \theta) \big|_{x=x^\star_\circ, \theta=0} \theta,
\end{multline}
is a first-order approximation of a mapping that zeros the gradient on a neighborhood of the origin.
\QEDB\end{remark}

Although the conditions in Proposition~\ref{prop:suff_cond_parameter_feedback} are immediate to verify, the convergence claims of
Theorem~\ref{thm:solvability_parameter_feedback} and 
Proposition~\ref{prop:suff_cond_parameter_feedback} are of local nature; 
namely, $y(t) \rightarrow 0$ is ensured provided that $\theta(0)$ is 
sufficiently close to the origin. The following result provides a sufficient 
condition for global convergence.

\begin{theorem}[\textbf{\textit{Existence of global  parameter-feedback algorithms}}]
\label{thm:suff_cond_parameter_feedback_global}
Let Assumptions~\ref{as:convexity_lipschitz_f0}--\ref{as:exosystem} hold.
Further, assume: 
\begin{enumerate}
\item[(i)] the loss function $f$ is $C^1$ on $\real^n \times \Theta$, 
\item[(ii)] $x \mapsto f(x,\theta)$ is $C^3$ on $\real^n$ for each $\theta \in \Theta$, 
\item[(iii)] the Hessian $\nabla^2_{xx} f(x, \theta) $ is positive definite on $\real^n \times \Theta$, 
\item[(iv)] the mappings $\nabla_{xx}^2 f$ and $ \frac{\partial \nabla_{xx}^2 f}{\partial x}$ are $C^0$ on $\real^n \times \Theta.$
\end{enumerate}
Then, there exits a $C^0$ mapping $\map{H_c}{\Theta}{\real^n}$ that zeros the gradient everywhere in $\Theta.$
\QEDB\end{theorem}

\begin{proof}
Define $F(x,\theta) = \nabla_x f(x,\theta)$ and note that, under the stated assumptions, $F$ is $C^0$ on $\real^n \times \Theta$, the mapping $x \mapsto F(x,\theta)$ is $C^2$ on $\real^n$ for all $\theta \in \Theta$,
\[
\det \left[ \frac{\partial F(x, \theta)}{\partial x} \right] \neq 0, \quad \forall (x,\theta)\in \real^n\times\Theta,
\]
and the mappings $\frac{\partial F}{\partial x}$ and $\frac{\partial^2 F}{\partial x^2}$ are $C^0$ in $\real^n \times \Theta$. Hence, assumptions (B1)--(B5) of a global version of the implicit function theorem~\cite[Thm.~2]{HN-FF:22} (see also~\cite[Thm.~6]{AA-SZ:19}) are satisfied for the first-order optimality conditions $0 = F(x, \theta)$, and the claim follows. 
\end{proof}

Theorem~\ref{thm:suff_cond_parameter_feedback_global} shows that, under additional continuity assumptions on the loss function  and when the Hessian of the loss is positive definite everywhere, there exists a mapping $\map{H_c}{\Theta}{\real^n}$ that zeros the gradient everywhere in $\Theta.$

By combination of Theorem~\ref{thm:suff_cond_parameter_feedback_global} and
Theorem~\ref{thm:solvability_parameter_feedback}, it follows that under the 
assumptions of Theorem~\ref{thm:suff_cond_parameter_feedback_global}, the parameter-feedback algorithm $x(t) = H_c(\theta(t))$ ensures that $y(t) \rightarrow 0$ as $t \rightarrow \infty$ for all initial conditions $\theta(0) \in \Theta$. This result ensures the existence of parameter-feedback algorithms for the particular class of strictly convex loss functions~\cite{MF-SP-VP-AR:17,bastianello2024}.

\section{The dynamic gradient-feedback problem}
\label{sec:gradient_feedback}

In the previous section, we analyzed algorithms assuming knowledge of the function 
$\nabla_x f(x,\theta)$ and of the signal $\theta(t), t \in \realnneg$ 
(cf. assumptions (E1) and (O2)). 
In this section, we relax these assumptions and require only 
that $y(t)$ is available through oracle evaluations of~\eqref{eq:gradient}.
For this reason, we will shift our attention back to the general class of 
dynamic gradient-feedback algorithms~\eqref{eq:controller_equatons_zx_f0}.

Because $\theta(t)$ is unmeasurable by the algorithm, the temporal variability 
of the cost can only be evaluated through measurements of $y(t);$ for this 
reason, we impose the following.

\begin{assumption}[\textbf{\textit{Detectability of the exosystem}}]
\label{as:detectability_exosystem}
Let
\begin{align}\label{eq:jacobians}
Q := \left[\frac{\partial \nabla_x f}{\partial \theta}\right]_{(x,\theta)=(x_\circ^\star,0)}, &&
S := \left[\frac{\partial s}{\partial \theta}\right]_{\theta=0}.
\end{align}
The pair $(Q,S)$ is detectable. 
\QEDB\end{assumption}

The detectability property of 
Assumption~\ref{as:detectability_exosystem} is equivalent to the existence of 
a local exponential observer for all (non-asymptotically stable) modes of 
$\theta(t)$ based on measurements $y(t)$~\cite[Cor.~3.4]{WL-CB:94}.
Undetectability corresponds to a redundant description of the exogenous 
signal: if some modes of $\theta(t)$ are undetectable, these modes do not 
influence the gradient (at $(x_\circ^\star,0)$) and, as a consequence, 
the critical points of~\eqref{eq:optimization_objective_f}; therefore, they can be removed from 
the optimization problem~\eqref{eq:optimization_objective_f} without altering 
its critical points.
In other words, Assumption~\ref{as:detectability_exosystem} ensures 
that~\eqref{eq:exosystem} does not contain redundancies.

\subsection{Existence and characterization of dynamic gradient feedback algorithms}

We begin by considering the class of 
algorithms~\eqref{eq:controller_equatons_zx_f0}, restricted to ensure local 
exponential stability, and we characterize a set of necessary and sufficient 
conditions for this class to achieve exact asymptotic tracking.

\begin{theorem}[\textbf{\textit{Gradient-feedback algorithm characterization}}]
\label{thm:internal_model_principle}
Suppose Assumptions~\ref{as:convexity_lipschitz_f0}--\ref{as:detectability_exosystem} hold, and assume that $F_c(z,y)$ and 
$G_c(z)$ are such that the equilibrium $z=z^\star_\circ$~of 
\begin{align}\label{eq:exponential_equilibrium}
\dot z(t) &= F_c(z(t), \nabla_x f(G_c(z(t)),0)),
\end{align}
is locally exponentially stable. The optimization
algorithm~\eqref{eq:copled_system_gradient_feedback} exactly asymptotically 
tracks a critical trajectory of~\eqref{eq:optimization_objective_f} if and only 
if 
there exists a neighborhood $\Theta_\circ \subset \Theta$ of the origin
and a $C^2$ mapping $z = \sigma(\theta)$ with $\sigma(0)=z_\circ^\star$, such 
that:
\begin{subequations}\label{eq:dynamic_condition}
\begin{align}
\left. \frac{\partial \sigma (\theta)}{\partial \theta} \right\vert_{\theta=\theta_\omega} s(\theta_\omega) &= F_c(\sigma(\theta_\omega),0),\label{eq:dynamic_condition_a}\\
0 &= \nabla_x f(G_c(\sigma(\theta_\omega)),\theta_\omega),\label{eq:dynamic_condition_b}
\end{align}
at all limit points $\theta_\omega \in \Omega(\Theta_\circ)$.
\end{subequations}
\end{theorem}

\begin{proof}
\textit{(Only if)}. We first prove that $\lim _{t \rightarrow \infty} y(t) = 0$ implies~\eqref{eq:dynamic_condition}.
The coupled dynamics~\eqref{eq:copled_system_gradient_feedback} have the form:
\begin{align}\label{eq:first_order_proof}
\dot z &= (A_c+B_cRM) z + B_cQ \theta + \chi(z, \theta), \nonumber\\
\dot \theta &= S \theta + \psi(\theta),
\end{align}
for some mappings $\chi(x, \theta)$ and $\psi(\theta)$ that vanish at the origin along with their first-order derivatives, with $Q$ and $S$ are defined in~\eqref{eq:jacobians}, and
\begin{align*}
A_c &= \left[\frac{\partial F_c}{\partial z}\right]_{(z,y)=(z_\circ^\star,0)}, &
B_c &= \left[\frac{\partial F_c}{\partial y}\right]_{(z,y)=(z_\circ^\star,0)}, \\
R&= \left[\frac{\partial \nabla_x f}{\partial x}\right]_{(x,\theta) = (x_\circ^\star,0)}, &
M&= \left[\frac{\partial G_c}{\partial z}\right]_{z=z_\circ^\star}.
\end{align*}
By assumption, the eigenvalues of the matrix $A_c+B_cRM$ are in 
$\complex^-$. By Theorem~\ref{thm:existence_center_manifold}, the 
system~\eqref{eq:first_order_proof} has a center manifold at 
$(z^\star_\circ,0)$: the graph of a mapping 
$z = \sigma(\theta),$ with $\sigma(\theta)$ satisfying (see~\eqref{eq:differential_equation_center_manifold})
\[
    \sigma ( s(\theta)) = F_c(\sigma(\theta),\nabla_x f(G_c(\sigma(\theta)),\theta)).
\]
Next, similarly to the parameter-feedback case (Theorem~\ref{thm:solvability_parameter_feedback}), by 
Assumption~\ref{as:exosystem}, $\theta(t_i)$ converges to some limit point 
$\theta_\omega \in \Omega\left(\theta(0)\right)$ for some subsequence
$\{t_i\}_{i\in\naturalnneg}$. 
By continuity of the gradient (Assumption~\ref{as:convexity_lipschitz_f0}) and 
that of $G_c$,
\begin{align}\label{eq:limit_y_sequence_2}
\lim_{i\to\infty} y({k_i})
    = \lim_{i\to\infty} \nabla_x f(G_c(\theta_{t_i}),\theta_{t_i})
    = \nabla_x f(G_c(\theta_\omega),\theta_\omega).
\end{align}
Because, $y(t)\to 0$, the left-hand side of~\eqref{eq:limit_y_sequence_2} is 
equal to zero, and thus 
$\nabla_x f(G_c(\theta_\omega),\theta_\omega)=0.$ Since this holds for any 
limit point $\theta_\omega\in \Omega(\theta_0)$, the claim follows.

\textit{(If).} 
We now prove that~\eqref{eq:dynamic_condition} implies 
${\lim _{t \rightarrow \infty} y(t) = 0}$.
It follows from~\eqref{eq:dynamic_condition} that the graph of the mapping 
$z = \sigma(\theta)$ (i.e., $(\sigma(\theta), \theta)$) is a center manifold
for the coupled dynamics~\eqref{eq:copled_system_gradient_feedback}. Hence, 
by Theorem~\ref{thm:local_attractivity_manifold},
such manifold is locally attractive; namely, 
$z(t) \rightarrow \sigma(\theta(t))$ as $t \rightarrow \infty.$
Then, the fulfillment of~\eqref{eq:dynamic_condition_b} guarantees that  
the right-hand side of~\eqref{eq:limit_y_sequence_2}
$y(t) \rightarrow 0.$ The conclusion then follows by iterating the 
\textit{(If)} part of the proof of 
Theorem~\ref{thm:solvability_parameter_feedback}.
\end{proof}

The two conditions in~\eqref{eq:dynamic_condition} fully characterize the 
class of optimization algorithms that achieve exact asymptotic tracking: 
\eqref{eq:controller_equatons_zx_f0} tracks a critical trajectory if 
and only if, for some mapping $\sigma,$ the composite function $G_c \circ \sigma$  zeros the gradient locally (see
\eqref{eq:dynamic_condition_b}), and the controller $F_c(z,y)$ is 
algebraically related to the exosystem $s(\theta)$ as given 
by~\eqref{eq:dynamic_condition_a}. 
Notice that, by Theorem~\ref{thm:solvability_parameter_feedback}, the former 
condition implies that
\begin{align}\label{eq:Gc_composite_sigma}
x(t) = G_c(\sigma(\theta)),
\end{align}
is a parameter-feedback optimization algorithm 
for~\eqref{eq:optimization_objective_f}.
Finally, we note that the existence of an exponentially stable system as 
in~\eqref{eq:exponential_equilibrium} is always guaranteed under 
Assumption~\ref{as:detectability_exosystem}~\cite[Sec.~3]{WL-CB:94}.

The following interpretation follows from our findings.

\begin{remark}[\textbf{\textit{The internal model principle}}]
\label{rem:Minternal_model_principle}
We refer to condition~\eqref{eq:dynamic_condition_a} as the \textit{internal 
model principle of time-varying optimization}, as it encapsulates the 
requirement that an optimization algorithm must incorporate an internal model 
of the exosystem to achieve exact tracking.
Note that the use of a copy of the temporal variability of the optimization 
problem is explicit in the prediction-correction algorithm (see~\eqref{eq:pred_corr_FcGc}).
\QEDB\end{remark}

Theorem~\ref{thm:internal_model_principle} allows us to re-interpret the basic 
gradient-flow algorithm as follows.

\begin{remark}[\textbf{\textit{Internal model-based interpretation of basic gradient flow methods}}]
\label{rem:internal_model_grad_flow}
Recall that the gradient-flow algorithm~\eqref{rem:static_gradient} can be 
viewed as an instance of~\eqref{eq:controller_equatons_zx_f0_b} with 
$F_c(z,y)$ and $G_c(z)$ as in Remark~\ref{rem:special_cases_optimization_alg}. 
By direct substitution into~\eqref{eq:dynamic_condition}, it is evident that this algorithm 
satisfies~\eqref{eq:dynamic_condition_a} with 
$s(\theta) = 0$ and $\sigma(\theta)$ arbitrary. 
Since $s(\theta) = 0$ describes the internal model of a constant signal, by 
Theorem~\ref{thm:internal_model_principle}, these algorithms achieve exact asymptotic tracking 
if and only if $\theta(t)$ is a constant signal \blue{(under 
Assumptions~\ref{as:convexity_lipschitz_f0}--\ref{as:detectability_exosystem})}. This observation is in line 
with the inexact convergence properties of these algorithms typically given for 
these algorithms~\cite{PAA-KK:06,EH:16,GB-JC-JP-ED:21-tcns}.~
\QEDB\end{remark}

Since Theorem~\ref{thm:internal_model_principle} shows that the problem of
exact asymptotic tracking can be reduced to finding an optimization algorithm
that is algebraically equivalent to the exosystem together with a mapping that
zeros the gradient, we conclude the following.

\textit{Answer to Problem~\ref{prob:assumptions}:}
When an algorithm as in~\eqref{eq:controller_equatons_zx_f0} that achieves 
exact asymptotic tracking exists, the minimal knowledge needed to design such an algorithm is given by assumptions (E2) and (O2).
Notice that assumption (O1) is implicitly required to execute the algorithm~\eqref{eq:controller_equatons_zx_f0}.

By comparison with the \textit{Answer to 
Problem~\ref{prob:minimal_knowledge_parameter}}, it follows that relaxing 
the knowledge of $\theta(t)$ (i.e., relaxing (E1)) comes at the cost of 
requiring knowledge of the exosystem (i.e., requiring (E2)).

By Theorem~\ref{thm:internal_model_principle},
the exosystem state $\theta$ and that of the optimization $z$ must be related, 
everywhere in $\Theta_\circ,$ by
\begin{align}\label{eq:change_of_coordinates}
z(t) =\sigma(\theta(t)).
\end{align}
Intuitively, \eqref{eq:change_of_coordinates} is interpreted as the existence 
of a change of coordinates between the state of the exosystem and that of the 
optimization; see Section~\ref{sec:injectivity_sigma} for a discussion on the 
invertibility properties of the function $\sigma(\theta)$.
%


\begin{remark}[\textbf{\textit{Special case with $\sigma$ being the identity operator}}]
\label{rem:sigma_identity}
An important special case is obtained when $\sigma$ is the identity operator 
on $\Theta$; in this case, \eqref{eq:dynamic_condition} 
simplifies to $s(\theta) = F_c(\theta,0)$,
which states that the controller vector field $F_c(z,y)$ must coincide with 
that of the exosystem $s(\theta)$ in $\Omega(\Theta_\circ).$ In this case, 
\eqref{eq:change_of_coordinates} gives $z(t) =\theta(t);$ namely, the 
controller state $z(t)$ and that of the exosystem $\theta(t)$ coincide in 
$\Omega(\Theta_\circ).$
\QEDB\end{remark}

While Theorem~\ref{thm:internal_model_principle} provides a full 
characterization of all gradient-feedback algorithms that achieve tracking, it 
remains to address under what conditions such an algorithm is guaranteed to 
exist.  This question is addressed next.

\begin{theorem}[\textbf{\textit{Existence of gradient-feedback algorithms}}]
\label{thm:solvability_gradient_feedback}
Suppose Assumptions~\ref{as:convexity_lipschitz_f0}--\ref{as:detectability_exosystem} hold.
Then, there exists a gradient-feedback algorithm, such that  $z=z^\star_\circ$
is exponentially stable for~\eqref{eq:exponential_equilibrium}, that solves 
Problem~\ref{prob:grad_feedback} if and only if there exists a mapping 
$H_c : \Theta\to\real^n$ that zeros the gradient on the limit set of~\eqref{eq:exosystem} with respect to its initializations. 
\QEDB\end{theorem}

\begin{proof}
\textit{(Only if)} By Theorem~\ref{thm:internal_model_principle}, there 
exists a mapping $z = \sigma(\theta)$ such that~\eqref{eq:dynamic_condition_b} 
holds. Then, \eqref{eq:gradient_condition} holds immediately by 
letting $H_c(\theta) = G_c(\sigma(\theta))$.

\textit{(If)} We will prove this claim by constructing a gradient-feedback 
algorithm that achieves $y(t) \rightarrow 0$ as $t\to\infty$. 
First, notice that by Assumption~\ref{as:detectability_exosystem}, there 
exists a matrix $L$ such that $S-LQ$ has eigenvalues in $\complex^-$.
Consider the 
algorithm~\eqref{eq:controller_equatons_zx_f0} with $n_c = p$ and\footnote{Note 
that the second argument of the gradient is evaluated at the algorithm state 
$z$ instead of the parameter vector $\theta$ as the latter is unknown.}
\begin{subequations}\label{eq:Fc_proof}
    \begin{align}
    F_c(z,y) &= s(z) + L (y - \nabla_x f (H_c(z), z)), \\
    G_c(z) &= H_c(z),
    \end{align}
\end{subequations}
where $H_c(z)$ is as in~\eqref{eq:gradient_condition}.
The Jacobian of $\dot z(t) = F_c(z(t), y(t))$ with respect to $z$ is given by $S-LQ$ (two terms of the form $LRM$ with opposite 
sign cancel out in forming the Jacobian: one from the gradient term 
$\nabla_x f (H_c(z), z)$ and the other one from $y$);
since $S-LQ$ has eigenvalues in $\complex^-$, $z=z^\star_\circ$ is exponentially stable for~\eqref{eq:exponential_equilibrium}.
The claim thus follows by
Theorem~\ref{thm:internal_model_principle} with $\sigma$ the identity operator on $\Theta$.
\end{proof}

Interestingly, the conditions for the existence of a gradient-feedback 
algorithm and those for a parameter-feedback algorithm (cf. 
Section~\ref{sec:parameter_feedback}) are identical. This should not be surprising: 
on the one hand, a parameter-feedback algorithm assumes access to $\theta(t)$; on the 
other hand, a gradient-feedback algorithm must infer $\theta(t)$ through $y(t)$, but 
(see~\eqref{eq:change_of_coordinates}) the dynamic state of the controller $z(t)$ acts as a 
copy (possibly in different coordinates) of $\theta(t)$. In this sense, both algorithm classes 
rely on essentially  the same information, and it is therefore natural that the conditions for 
their existence coincide.


\subsection{Design of dynamic gradient feedback algorithms}

We now focus on addressing Problem~\ref{prob:grad_feedback}.
A design procedure to construct $F_c(z,y)$ and $G_c(z)$ is presented in 
Algorithm~\ref{alg:grad_feedback_design}, which is derived from the proof of 
Theorem~\ref{thm:solvability_gradient_feedback}. 

The algorithm uses a Luenberger observer to estimate the exosystem state 
$\theta(t)$ (cf. line $4$), and a parameter feedback algorithm is applied 
to the estimated exosystem state to regulate the gradient to zero; precisely, 
$G_c(z)$ is designed following the approach of 
Theorem~\ref{thm:solvability_parameter_feedback} (see line $3$ of 
Algorithm~\ref{alg:grad_feedback_design}). 
Notice that the algorithm uses the particular choice for $\sigma(\theta)$ being 
the identity operator (see Remark~\ref{rem:sigma_identity}).
Note that line $2$ of the algorithm guarantees that the equilibrium 
of~\eqref{eq:exponential_equilibrium} is locally exponentially stable, as 
required by Theorem~\ref{thm:internal_model_principle}, and that the 
existence of such $L$ is guaranteed by 
Assumption~\ref{as:detectability_exosystem}.

\begin{algorithm}
\caption{Gradient-feedback algorithm design}
\label{alg:grad_feedback_design}
\KwData{$s(\theta),$ $\nabla_x f(x,\theta),$ $G_c(\theta)$ satisfying~\eqref{eq:dynamic_condition_b}, Jacobian matrices $Q$ and $S$ in~\eqref{eq:jacobians}}
$n_c \gets p$\;
$L \gets$ any matrix such that $S-LQ$ is Hurwitz\;
$G_c(z) \gets H_c(z)$\;
$F_c(z,y) \gets s(z) + L (y - \nabla_x f (H_c(z), z))$\;
\KwResult{$F_c(z,y)$, $G_c(z)$, and $n_c$ that solve Problem~\ref{prob:grad_feedback}}
\end{algorithm}


\begin{theorem}[\textbf{\textit{Correctness of Algorithm~\ref{alg:grad_feedback_design}}}]
\label{thm:algorithm_regulation}
Suppose Assumptions~\ref{as:convexity_lipschitz_f0}--\ref{as:detectability_exosystem} hold. Then, the control algorithm~\eqref{eq:controller_equatons_zx_f0} designed 
using Algorithm~\ref{alg:grad_feedback_design} asymptotically tracks a 
critical trajectory of~\eqref{eq:optimization_objective_f}.
\end{theorem}

\begin{remark}[\textbf{\textit{Alternative observer choices}}]
Instead of a Luenberger observer, alternative dynamic observers could be 
considered in Line 4 of the algorithm to achieve different asymptotic or 
transient properties of the resulting gradient-feedback optimization algorithm; 
we leave the investigation of alternative state observer algorithms to future works.
\QEDB\end{remark}

We illustrate the algorithm design procedure on a quadratic problem in the 
following example.

\begin{example}
\label{ex:quadratic_gradient_feedback}
Consider the quadratic problem studied in 
Example~\ref{ex:quadratic_cost}, and assume that the exosystem follows 
the linear model $\dot \theta = S \theta$ for some matrix 
$S \in \real^{p\times p}$ as in Assumption~\ref{as:exosystem}.
Further, suppose that $Q$ is such that the unobservable subspace
$\mc N = \cap_{i \geq 0} \ker(QS^i)$ is the empty set $\mc N = \emptyset,$ so 
that Assumption~\ref{as:detectability_exosystem} is satisfied; notice that this 
assumption holds, for example, when $n=p$ and $Q$ is invertible. 
According to Theorem~\ref{thm:solvability_gradient_feedback}, an
optimization algorithm given by
\begin{align}\label{eq:controller_linear}
    \dot z = A_c z + B_c y, && x = G_c z,
    && y = Rx + Q \theta,
\end{align}
where $A \in \real^{n_c \times n_c},$ $B_c \in \real^{n_c \times n},$
$G_c \in \real^{n \times n_c},$ achieves asymptotic tracking if and 
only if there exists a linear transformation 
$\Sigma \in \real^{n_c \times p}$ such that:
\begin{subequations}\label{eq:dynamic_condition_linear}
\begin{align} 
\Sigma S &= A_c \Sigma, \label{eq:dynamic_condition_linear_a}\\
0 &= R G_c \Sigma + Q. \label{eq:dynamic_condition_linear_b}
\end{align}
\end{subequations}


Next, we apply Algorithm~\ref{alg:grad_feedback_design} to design an algorithm 
that solves this problem. The mapping $H_c(\theta) = - R^\dagger Q \theta$ zeros the gradient (see Example~\ref{ex:quadratic_cost}). Following Algorithm~\ref{alg:grad_feedback_design}, we set $n_c = p$, choose matrix $L$ such that $S-LQ$ is Hurwitz, define $G_c(z) = H_c(z) = -R^\dagger Q z$, and then select $A_c$ and $B_c$ such that $F_c(z,y) = A_c z + B_c y$ satisfies:
\begin{align*}
F_c(z,y) &= s(z) + L (y - \nabla_x f (H_c(z), z)),\\
&= S z + L(y - (R H_c(z) + Qz)),\\
&= S z + L y,
\end{align*}
which gives $A_c = S$ and $B_c = L.$ 
In summary, a dynamic gradient-feedback algorithm is~\eqref{eq:controller_linear} with $A_c = S$, $B_c = L$, and $G_c = -R^\pinv Q$,
where $L$ is any matrix such that $S-LQ$ is Hurwitz.
With this choice, conditions~\eqref{eq:dynamic_condition_linear} are satisfied 
with $\Sigma = I$ (see Remark~\ref{rem:sigma_identity}). 
Finally, the autonomous model \eqref{eq:exponential_equilibrium} 
becomes $\dot z(t) = (S - LQ) z(t)$,
which is asymptotically stable, as required by 
Theorem~\ref{thm:internal_model_principle}.
\QEDB\end{example}

\subsection{The optimization algorithm as a copy of the exosystem}
\label{sec:injectivity_sigma}

As shown in the previous section, the state of the optimization algorithm and 
that of the exosystem must be related, at all points of the limit set, by 
\eqref{eq:change_of_coordinates}. In this section, we show that
$\sigma(\theta)$ is injective, and therefore defines a well-posed change of 
coordinates. We begin by establishing the following result.

\begin{proposition}[\textbf{\textit{Injectivity of $\sigma(\theta)$ from Algorithm~\ref{alg:grad_feedback_design}}}]
\label{prop:alg_sol_is_injective}
Let the assumptions of Theorem~\ref{thm:solvability_gradient_feedback} hold and 
let $F_c(z,y)$ and $G_c(z)$ be obtained by 
Algorithm~\ref{alg:grad_feedback_design}. Then, there exists an injective map 
$\sigma(\theta)$ such that~\eqref{eq:change_of_coordinates} holds. 
\QEDB\end{proposition}

\begin{proof}
The claim follows by noting that $F_c(z,y)$ and $G_c(z)$ obtained by 
Algorithm~\ref{alg:grad_feedback_design} satisfy~\eqref{eq:dynamic_condition}
with $\sigma$ the identity. 
\end{proof}

It follows from Proposition~\ref{prop:alg_sol_is_injective} that, for 
any optimization method returned by Algorithm~\ref{alg:grad_feedback_design}, 
there exists an injective mapping that, locally, maps the state of the 
exosystem into the state of the optimization algorithm. 
This property is not exclusive to Algorithm~\ref{alg:grad_feedback_design}; in 
fact, any algorithm that achieves exact asymptotic tracking possesses the same 
characteristic, as demonstrated next.

\begin{proposition}[\textbf{\textit{Injectivity of any $\sigma(\theta)$}}]
\label{prop:sigma_injective}
Let the assumptions of Theorem~\ref{thm:solvability_gradient_feedback} hold, 
and assume that $s(\theta)$ is of class $C^\infty$. Then, for any algorithm of 
the form~\eqref{eq:controller_equatons_zx_f0} that achieves asymptotic 
tracking, there exists a neighborhood $\Theta_\circ \subset \Theta$ of the 
origin such that $\sigma(\theta)$ is injective on $\Theta_\circ$.~\QEDB
\end{proposition}

\begin{proof}
By contradiction, assume that $F_c(z,y)$ and $G_c(z)$ 
satisfy~\eqref{eq:dynamic_condition}, but $\sigma(\theta)$ is not injective at 
the origin; namely, there exists nonzero $\theta'\in\Theta_\circ$ such that 
$\sigma(\theta')=z_0^\star.$ Then from~\eqref{eq:dynamic_condition_b},
$0 = \nabla_x f(G_c(\sigma(\theta')), \theta')$.
Moreover, from Theorem~\ref{thm:internal_model_principle},
\begin{align}\label{eq:aux_grad}
    \frac{\partial \sigma(\theta')}{\partial \theta} s(\theta') &= F_c(\sigma(\theta'),0) = 0.
\end{align}
For $C^\infty$ vector fields $h_1(x)$ and $h_2(x)$, we let $L_{h_1}(h_2)(x) = \frac{\partial h_2(x)}{\partial x} h_1(x).$ 
By application of \eqref{eq:aux_grad}, we have
\begin{align*}
0 = L_s(\nabla_x f) (x,\theta') = L_s( L_s(\nabla_x f))(x,\theta') = \dots . 
\end{align*}
Hence, the matrix of Lie derivatives of the measurable signal, which, 
intuitively, characterizes the information content of the gradient signal:
\begin{align*}
\begin{bmatrix}
L_s(\nabla_x f) (x,\theta)\\
L_s(L_s(\nabla_x f)) (x,\theta)\\
\vdots 
\end{bmatrix},
\end{align*}
is not invertible at $\theta=\theta'.$ By~\cite[Thm.~3.13]{RH-AK:77},
the system~\eqref{eq:exosystem} is not weakly observable, thus 
violating Assumption~\ref{as:detectability_exosystem}.
\end{proof}


\section{Fidelity of the internal model vs accuracy}
\label{sec:fidelity_model}

In Section~\ref{sec:gradient_feedback}, we showed that a necessary condition for 
exact asymptotic tracking is that the algorithm incorporates an internal model 
of the temporal variability of the problem.
Departing from this finding, in this section we tackle the following question:
\textit{how will the algorithm perform when the temporal variability of the problem is 
known only approximately?}
For illustration purposes, we focus on the quadratic 
problem~\eqref{eq:quadratic_cost_example} (see  
Examples~\ref{ex:quadratic_cost} and 
\ref{ex:quadratic_gradient_feedback}).
Suppose that an imprecise internal model is available; that is, there exists 
$\Delta \in \real^{n_c \times p}$ such that 
\eqref{eq:dynamic_condition_linear} is modified~to\footnote{
To explicitly model uncertainty in the matrix $S,$ equation \eqref{eq:dynamic_condition_linear_approx} could be rewritten as
$\Sigma (S+\Delta_s) = A_c \Sigma.$ By comparison, the two approaches are 
equivalent with the substitution $\Delta = \Sigma  \Delta_s.$
}:
\begin{equation}\label{eq:dynamic_condition_linear_approx}
    \Sigma S + \Delta = A_c \Sigma \qquad\text{and}\qquad
    0 = R G_c \Sigma + Q.
\end{equation}

To analyze the asymptotic behavior of $y(t),$ it is useful to define the auxiliary 
variable $\tilde z(t) = z(t) - \Sigma \theta(t).$ 
Using~\eqref{eq:dynamic_condition_linear_approx} 
and~\eqref{eq:controller_linear}, the dynamics of $\tilde z$ follow the 
model:
\begin{align}\label{eq:z_tilde}
    \dot{\tilde z} = (A_c + B_cRG_c) \tilde z + \Delta \theta.
\end{align}
The corresponding gradient signal in terms of $\tilde z$ is $y = R G_c \tilde z$.
Assuming that $(A_c +B_cRG_c)$ is Hurwitz stable (see 
Theorem~\ref{thm:internal_model_principle}), the Final Value Theorem 
gives:
\begin{align}\label{eq:final_value_thm}
y_\infty &:= \lim_{t \rightarrow \infty} y(t)  \\
&= \lim_{s \rightarrow 0} s RG_c(sI - A_c -B_cRG_c)^\inv \Delta (sI-S)^\inv \theta(0). 
\nonumber
\end{align}
As a first illustrative scenario, suppose $S=0;$ namely, that the 
exosystem states are constants at all times. Notice that, since $S=0,$ 
from~\eqref{eq:dynamic_condition_linear_a}, $A_c$ incorporates an 
internal model of $S$ if and only if $A_c \Sigma=0;$ precisely, 
if and only if the dimension of the kernel of $A_c$ is at least $p$. In this case, assuming that the perturbed $A_c$ is such that
$(A_c +B_cRG_c)$ remains Hurwitz stable, from \eqref{eq:final_value_thm}, 
we have $y_\infty  = -RG_c(A_c +B_cRG_c)^\inv \Delta \theta(0)$.
As expected, when $\Delta =0,$ $A_c$ incorporates an exact internal 
model, and $y_\infty=0$ for any $\theta(0).$ 
In contrast, when $\Delta \neq 0,$
\begin{align*}
\norm{y_\infty} \leq  \norm{RG_c(A_c +B_cRG_c)^\inv}\norm{\theta(0)}\norm{\Delta}, 
\end{align*}
namely, bounded errors in $\Delta$ result in a bounded $y_\infty.$

Interestingly, analogous continuity properties may not hold when the 
trajectories of~\eqref{eq:exosystem} are not bounded (thus violating 
Assumption~\ref{as:exosystem}).
Consider an instance 
of~\eqref{eq:quadratic_cost_example} with $n=p=2,$  $R=Q=I,$  and $S = \left[\begin{smallmatrix} 0 & 1 \\ 0 & 0 \end{smallmatrix}\right]$. The solutions of~\eqref{eq:exosystem}~with this
exosystem are a linear combination of modes $e^{0t}$ and~$te^{0t},$ and thus grow unbounded at a linear rate.
Suppose algorithm~\eqref{eq:controller_linear} is used with 
$G_c=I$, $B_c = I,$ and $A_c = \left[\begin{smallmatrix}
    -\epsilon_1 & 1 \\ 0 & -\epsilon_2
\end{smallmatrix}\right]$,
where $\epsilon_1,\epsilon_2\in[0,1)$. 
When $\epsilon_1=\epsilon_2=0,$ 
\eqref{eq:dynamic_condition_linear_approx} holds with $\Sigma=-I$ and
$\Delta=0$, in which case $y_\infty =0$ and exact asymptotic tracking is 
achieved. 
On the other hand, when $\epsilon_1, \epsilon_2 \neq 0,$ 
\eqref{eq:dynamic_condition_linear_approx} holds with $\Delta = \left[\begin{smallmatrix}
    \epsilon_1 & 0 \\ 0 & \epsilon_2
\end{smallmatrix}\right]$
and, from~\eqref{eq:final_value_thm}, we obtain:
\begin{align*}
y_\infty = \lim_{s \rightarrow 0} 
\begin{bmatrix}
\frac{\epsilon_1}{s-\alpha_1}     & \frac{\epsilon_1}{s(s-\alpha_1)} + \frac{\epsilon_2}{(s-\alpha_1)(s-\alpha_2)}\\
0 & \frac{\epsilon_2}{s-\alpha_2}
\end{bmatrix} \bmat{\theta_1(0) \\ \theta_2(0)},
\end{align*}
where $\alpha_i = 1-\epsilon_i$, $i \in \{1, 2\}.$ 
Since $\epsilon_1 \neq 0$, and provided that $\theta_2(0)\neq 0$, the 
above identity implies that $\Vert y(t) \Vert \rightarrow \infty$ as 
$t \rightarrow \infty$, i.e., the asymptotic tracking error grows unbounded. 

This example illustrates that the tracking accuracy depends on the fidelity of 
the internal model but also the asymptotic behavior of the exosystem.

\section{Extensions}\label{sec:extensions}
In this section, we discuss possible extensions to constrained optimization problems and to discrete-time algorithms. 

\subsection{Constrained problems}

Our results can be extended to constrained time-varying 
optimization. Consider the equality-constrained problem
\begin{alignat*}{2}
    &\text{minimize}\quad && f(x,\theta(t)) \\
    &\text{subject to}\quad && h_i(x,\theta(t)) = 0, \quad i=1,\ldots,m,
\end{alignat*}
where the constraint functions $h_i(x,\theta(t))$ depend on the parameter vector. The associated Lagrangian function is
\[
    L(x,\lambda,\theta(t)) = f(x,\theta(t)) + \sum_{i=1}^m \lambda_i h_i(x,\theta(t)),
\]
where $\lambda_i$ is the Lagrange multiplier associated with the $i\textsuperscript{th}$ equality constraint. A pair $(x,\lambda)\in\real^n\times\real^m$ is said to be a saddle-point of the Lagrangian if, for all $(\bar x,\bar\lambda)\in\real^n\times\real^m$,
\[
    L(x,\bar\lambda,\theta(t)) \leq L(x,\lambda,\theta(t)) \leq L(\bar x,\lambda,\theta(t)). 
\]
For any such 
saddle-point, if strong duality holds, $x$ is primal optimal, $\lambda$ is dual 
optimal, and the optimal duality gap is zero. Moreover, the gradient of the Lagrangian (assuming it exists) is zero at any saddle-point. It follows from the derivations in the previous sections that the 
gradient-feedback and parameter-feedback algorithms can be directly applied to 
seek a stationary point of the Lagrangian function by replacing 
(in~\eqref{eq:optimization_objective_f}) the variable $x$ with the extended 
decision variable $\tilde x = (x, \lambda)$ and by letting 
$f(\tilde x, \theta) = L(x, \lambda, \theta).$
Notice that, if the critical point computed 
by~\eqref{eq:controller_equatons_zx_f0_a} is also a saddle-point and strong 
duality holds, then it is also a solution to the equality-constrained problem. When strong duality does 
not hold, a saddle-point may not correspond to an optimizer~\cite[Ch.~5]{boyd2004}.


\subsection{Discrete-time algorithms}
\label{sec:discrete-time}
In practice, optimization algorithms are often implemented on digital machines 
and thus require discretization of~\eqref{eq:controller_equatons_zx_f0}. 
It is worth noting that this scenario corresponds to designing a discrete-time 
algorithm to track a continuous-time optimizer 
(cf.~\eqref{eq:optimization_objective_f} and \eqref{eq:critical_trajectory}).
In these cases, a possible solution is to pursue a \textit{design by 
emulation} \cite{DN-AT-DC:09}, which consists in converting the continuous-time
algorithm~\eqref{eq:controller_equatons_zx_f0} into a 
discrete-time one by approximation (that is, by approximating the 
time-derivative by a discrete difference using, e.g., Backward/Forward 
Difference, Tustin, ZOH methods, etc.). Then, the continuous-time 
guarantees given here can be translated into discrete-time ones using 
standard tools from emulation~\cite{DN-AT-DC:09}, provided that the sampling 
time is sufficiently small. In these cases, 
exact tracking established above will hold at the discrete time instants, 
while tracking will be inexact in-between sampling times and the error can be 
bounded in terms of the sampling frequency using the explicit bounds 
in~\cite{DN-AT-DC:09}.

An alternative approach that could be pursued is that of discretizing the 
temporal variability $\theta(t)$, thus leading to a discrete-time optimization 
problem as in~\eqref{eq:discrete_optimization}.
For this alternative formulation, the critical trajectories are instead 
discrete signals and, in this case, it would be natural to seek a discrete 
algorithm, in place of~\eqref{eq:controller_equatons_zx_f0}, and pursue an 
analysis in  discrete time. Such a design is beyond the scope of this work and 
is left as the focus of future works.

\section{Simulation results}
\label{sec:simulations}

In this section, we illustrate our results and optimization design method
through a set of numerical simulations. For implementation on digital 
machines, all algorithms have been discretized by using the explicit 
Runge-Kutta-Fehlberg method of order $4(5)$ with adaptive stepsize~\cite{PB-FL:96}.


\begin{figure}[t]
\centering \subfigure{\includegraphics[width=\columnwidth]{
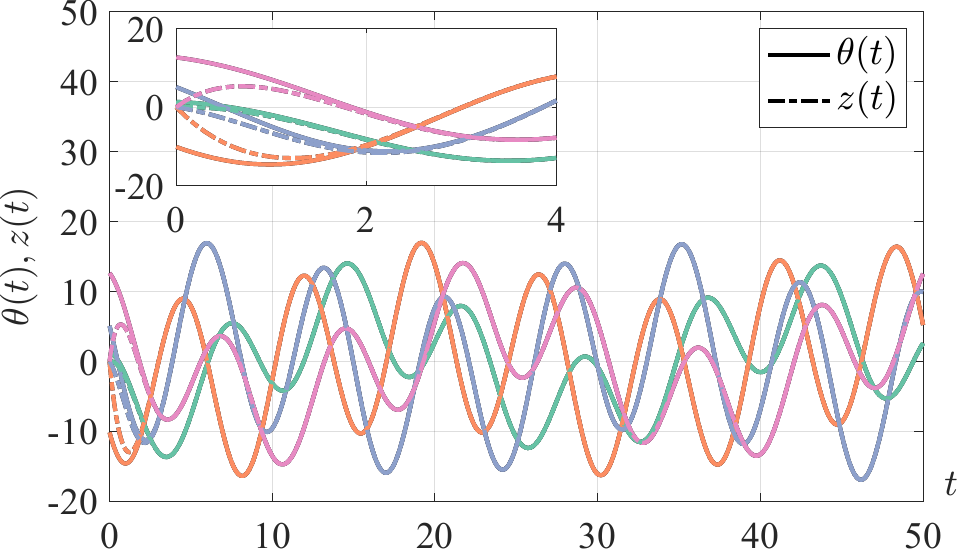}}
\centering \subfigure{\includegraphics[width=\columnwidth]{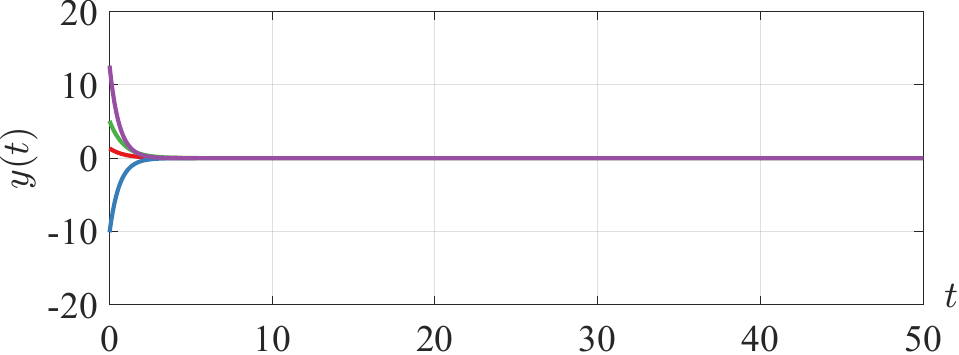}}
\centering \subfigure{\includegraphics[width=\columnwidth]{
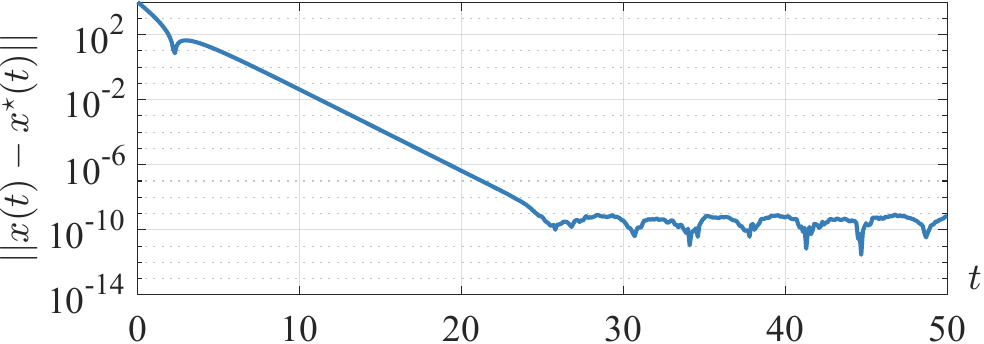}} 
\caption{%
Simulation results illustrating the performance of an optimization algorithm 
synthesized using Algorithm~\ref{alg:grad_feedback_design} for the quadratic 
instance~\eqref{eq:quadratic_cost_example} 
of~\eqref{eq:optimization_objective_f}. See Example~\ref{ex:quadratic_cost} and 
Section~\ref{sec:quadratic_cost_simulations} for a discussion. (Top) 
Illustration of the temporal variability of the parameter $\theta(t)$ and 
of $z(t).$ (Second from top) $z(t)$ is an estimator for $\theta(t),$ and thus 
$z(t) \rightarrow \theta(t)$ as $t \rightarrow \infty.$ (Third from top) The 
proposed control algorithm is successful in regulating the gradient feedback 
signal $y(t)$ to zero asymptotically. (Bottom) Illustration that 
$x(t) \rightarrow x^\star(t)$ as $t \rightarrow \infty.$
}
\vspace{-.5cm}
\label{fig:quadratic_simulation}
\end{figure}

\subsection{Optimization design for quadratic costs}
\label{sec:quadratic_cost_simulations}
We begin by numerically investigating the quadratic instance 
of~\eqref{eq:optimization_objective_f} with linear temporal variability, discussed previously in Examples~\ref{ex:quadratic_cost} and \ref{ex:quadratic_gradient_feedback}. With dimensions $n=p=4$, we chose the matrix $R\in\symmetric^4$ with random entries such that its eigenvalues are uniformly distributed in the open real interval $(0,1)$, and we set $Q = I \in\real^{4\times 4}$. We let the exosystem be $\dot\theta(t) = S\theta(t)$, where $S \in \real^{4 \times 4}$ is given by $S = \tilde S-\tilde S^\tsp$, and $\tilde S$ is a matrix with random entries uniformly distributed in the open real interval $(0,1)$. Notice that this choice ensures that the eigenvalues of $S$ are on the imaginary axis. We chose
$H_c(\theta)=H_c \theta$ with $H_c = -R^\inv Q,$ which 
ensures~\eqref{eq:gradient_condition} holds (see 
Example~\ref{ex:quadratic_cost}). We applied 
Algorithm~\ref{alg:grad_feedback_design}, choosing $L$ such the eigenvalues of 
$S-LQ$ are uniformly distributed in the real interval $(-2,-1)$,
as described in Example~\ref{ex:quadratic_gradient_feedback}.
Simulation results are presented in Fig.~\ref{fig:quadratic_simulation}. 
The top plot illustrates the 
time-varying signal $\theta(t)$ and the state of the optimization algorithm 
$z(t)$; the figure illustrates that $z(t)$ converges to $\theta(t)$ 
asymptotically (see discussion immediately after 
Theorem~\ref{thm:solvability_gradient_feedback}).
The middle plot illustrates the gradient signal, validating 
the claim that the algorithm converges to a critical point (where the gradient 
is zero). Finally, the bottom plot illustrates, in log scale, the error 
between the optimization state $x(t)$  and the optimizer $ x^\star(t)$; this 
figure shows that convergence to the optimizer is exponential, up to numerical 
precision.


\subsection{Nonlinear example}\label{sec:nonlinear-example}

We now illustrate our algorithm design on a nonlinear time-varying optimization problem from~\cite{AS-AM-AK-GL-AR:16,RR-AM-UV:22}, and we compare the results with several prediction-correction algorithm variants. Consider minimizing the function:
\[
    f(x,\theta) = \tfrac{1}{2} (x-\theta)^2 + \kappa \log(1+\exp(\mu x)),
\]
with parameters $\kappa = 7.5$ and $\mu = 1.5$, where the time-varying parameter $\theta(t)$ satisfies the exosystem dynamics $\dot\theta = S\theta$ where $S = \left[\begin{smallmatrix} 0 & 1 \\ 1 & 0\end{smallmatrix}\right]$. We compare the error $\|x(t) - x^\star(t)\|$ over time for several algorithms in Fig.~\ref{fig:example-nonlinear}. The objective function satisfies the conditions in Theorem~\ref{thm:suff_cond_parameter_feedback_global}; for our simulations, we computed numerically a mapping that zeros the gradient locally. In Fig.~\ref{fig:example-nonlinear}(a), we show the error for our algorithm with observer pole locations $(-\eta,-2\eta)$ for $\eta\in\{1,5,20\}$ (shown in dark to light blue) along with the corresponding rate $e^{-\eta t}$ (dashed black). For comparison, we also show the error of the prediction-correction (PC) algorithm~\eqref{eq:pred_correction} along with two of its finite-time (FT) variants~\cite[Equations (9) and (20)]{RR-AM-UV:22} using the parameter values in \cite[Section VI.A]{RR-AM-UV:22}.


Instead of computing $H_c(\theta)$ numerically, we can use its linear approximation in~\eqref{eq:Hc-linear}. Doing so yields the results in Fig.~\ref{fig:example-nonlinear}(b) for various size exosystem initial conditions $\theta(0) = \zeta\theta_0$ for $\zeta\in\{1,0.1,0.01\}$. The linear approximation does not zero the gradient exactly, which results in a finite steady-state error whose size depends on the size of the limit set of the exosystem (which corresponds to the size of the initial conditions since the exosystem dynamics are sinusoidal).

\begin{figure}
    \centering
    \subfigure[Error over time for the proposed dynamic gradient-feedback (DGF) algorithm with observer pole locations $(-\eta,-2\eta)$ for $\eta\in\{1,5,20\}$ (dark to light blue) along with the prediction-correction (PC) algorithm and two finite-time (FT) variants.]{\includegraphics[width=0.9\columnwidth]{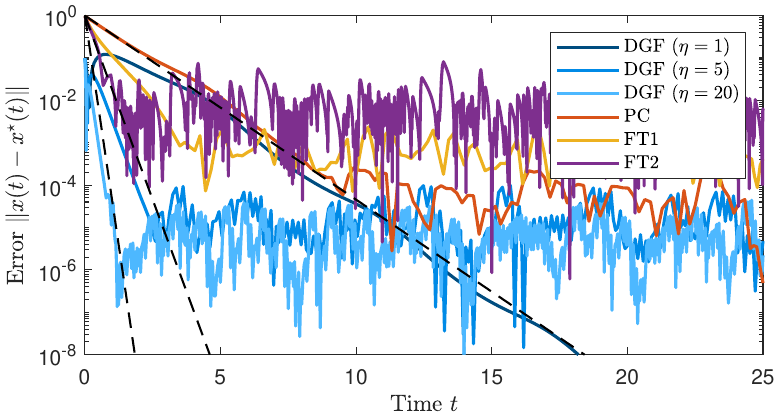}} \\
    
    \subfigure[Error over time for the proposed dynamic gradient-feedback algorithm using the linear approximation to $H_c(\theta)$ in~\eqref{eq:Hc-linear} with different size initial conditions $\theta(0) = \zeta\theta_0$ for $\zeta\in\{1,0.1,0.01\}$. When the initial conditions are small, the exosystem state stays near the equilibrium, in which case $H_c$ approximately zeros the gradient on the limit set of the exosystem.]{\includegraphics[width=0.9\columnwidth]{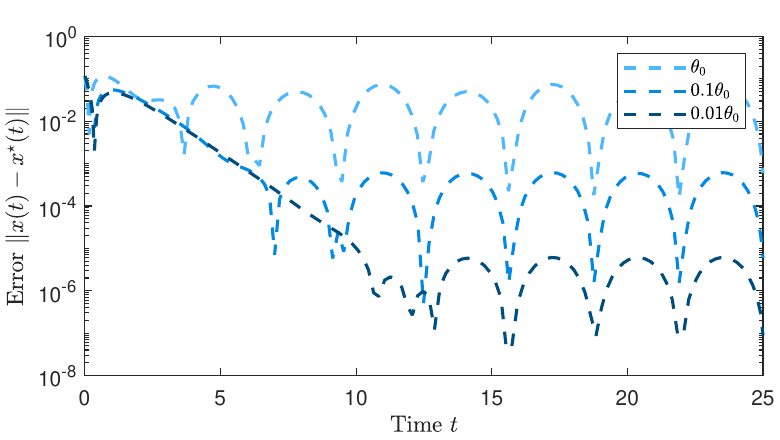}} \\
    \caption{Dynamic gradient-feedback algorithm applied to the nonlinear time-varying optimization problem from Section~\ref{sec:nonlinear-example}.}
    \label{fig:example-nonlinear}
\end{figure}

\subsection{Application to solve the dynamic traffic assignment problem in transportation}
\label{sec:dynamic_traffic_assignment}

\begin{figure}[t]
\centering
\subfigure[]{\includegraphics[width=.49\columnwidth]{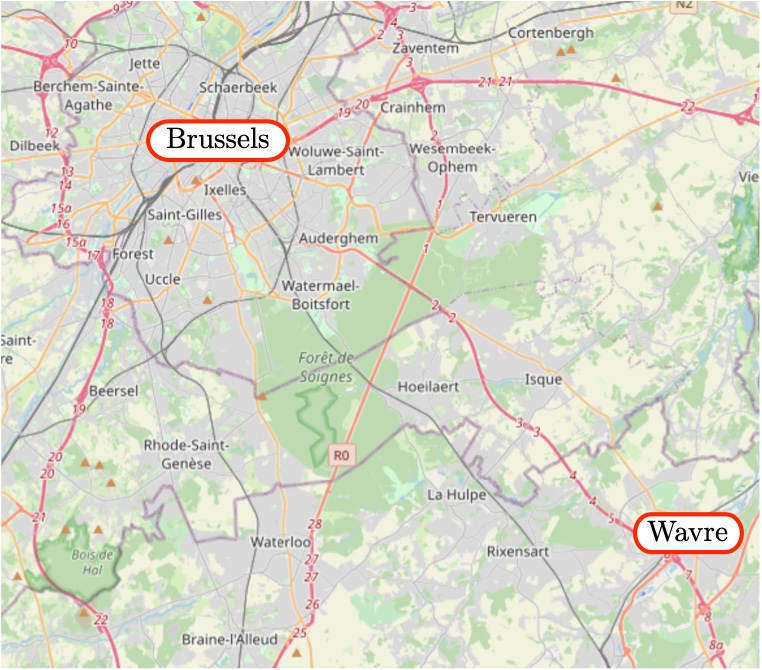}} \hfill
\subfigure[]{\includegraphics[width=.49\columnwidth]{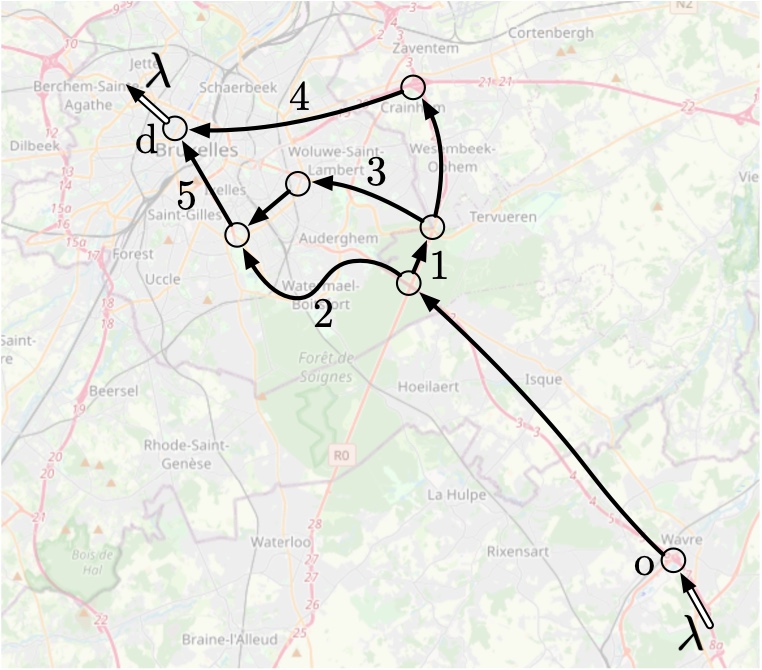}} \\
\caption{(Left) Areal view of the highway system between the cities of Wavre 
and Brussels, Belgium. (Right) Graph utilized to model the portion of traffic 
network of interest.}
\label{fig:traffic_network_brussels}
\vspace{-.5cm}
\end{figure}

\begin{figure}[t]
\centering \subfigure{\includegraphics[width=\columnwidth]{
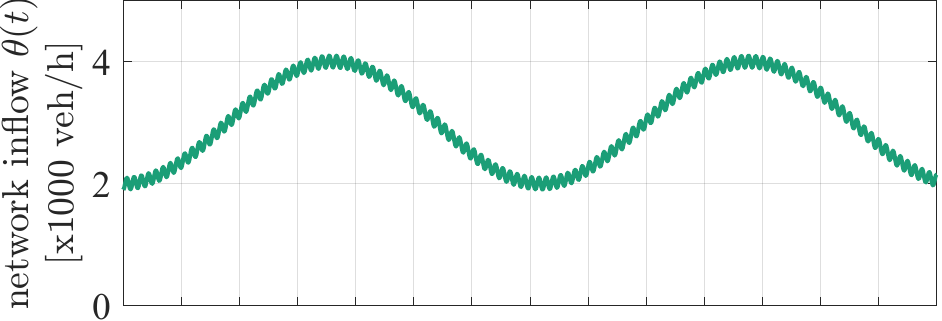}}
\centering \subfigure{\includegraphics[width=\columnwidth]{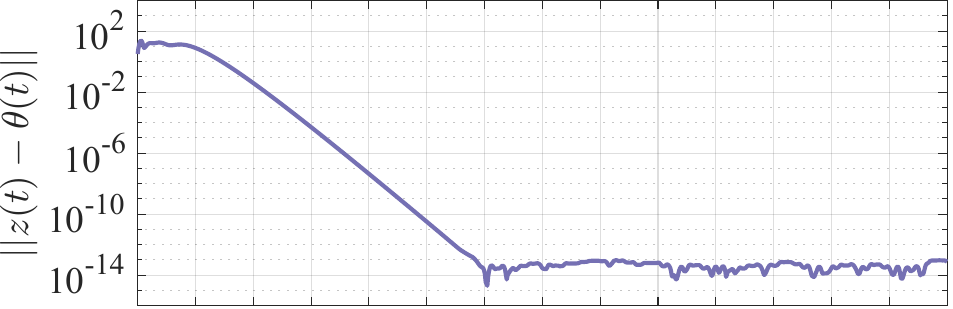}}
\centering \subfigure{\includegraphics[width=\columnwidth]{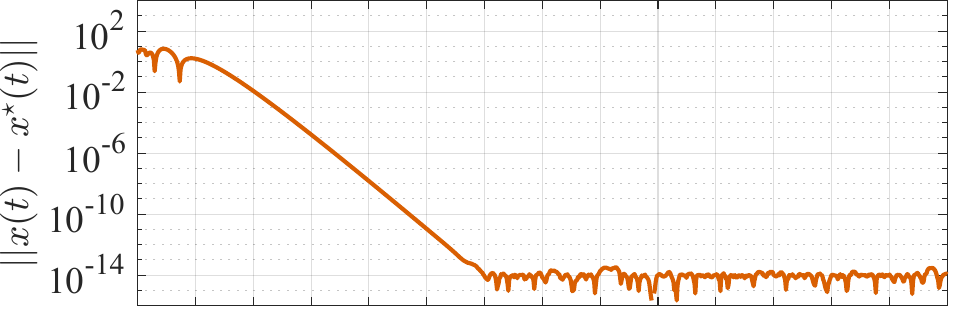}}
\centering \subfigure{\includegraphics[width=\columnwidth]{
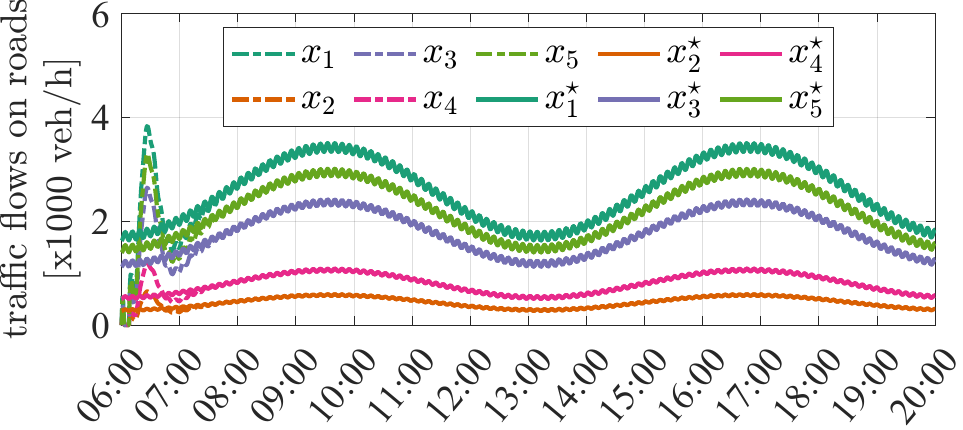}} 
\caption{%
Illustration of the performance of the optimization method 
synthesized using Algorithm~\ref{alg:grad_feedback_design} to solve the 
dynamic traffic assignment problem~\eqref{eq:traffic_assignment}. 
See Section~\ref{sec:dynamic_traffic_assignment} for a discussion. The 
bottom figure shows that, from a zero initial condition, the algorithm is 
capable of computing a Wardrop equilibrium in about $1.5$ hours, and then of 
tracking this equilibrium. 
Notice that, a faster rate of convergence could be 
obtained by shifting the eigenvalues of $S-LQ.$
}
\label{fig:wardrop_simulation}
\end{figure}

We next illustrate the applicability of the framework in solving the dynamic 
traffic assignment problem in roadway transportation~\cite{MP:15}; 
intuitively, the objective is to decide how traffic flows are split among the 
available paths of a network to minimize the drivers' travel time to 
destination.
We model a roadway transportation network using a static flow 
model~\cite{MP:15}, described by a directed graph 
$\mc G =(\mc V, \mc E),$ with edges $i \in \mc E \subseteq \mc V \times \mc V$ 
(modeling traffic roads) and nodes $\mc V$ (modeling traffic junctions). For 
$i \in \mc E,$ we denote by $i^+ \in \mc V$ and $i^- \in \mc V$ its origin and 
destination nodes, respectively.
We assume that an exogenous, time-varying, inflow of traffic $\theta(t)$ 
enters the network at a certain origin node, denoted by  $o \in \mc V,$ and 
exits at a certain destination node, denoted by $d \in \mc V;$
for simplicity, we assume that there is only one origin-destination 
pair, but this is without loss of generality\footnote{When the network has 
multiple origin-destination pairs, an optimization problem of the 
form~\eqref{eq:traffic_assignment} will be associated with each 
origin-destination pair. Because the optimization problems are independent from one 
another~\cite{MP:15}, each of them can be solved independently using the methods 
illustrated here.}. 
We describe the network state using a vector 
$x \in \realnneg^{\vert \mc E \vert}$ (where $\vert \mc E\vert$ denotes the number of edges) whose entries $x_i$ describe the amount of inflow 
$\theta$ routed through road $i.$ 
To each link $i,$ we associate a function $\ell_i(x_i)$ describing the  
latency (or travel time) of the road $i$. 
For our simulations, we consider the network topology in 
Fig.~\ref{fig:traffic_network_brussels}, and choose the latency functions as 
follows:
\begin{align*}
\ell_1(x_1) &= x_1, & \ell_2(x_2) &= 10x_2, & \ell_3(x_3) &= x_3, \\
\ell_4(x_4) &= 5x_4, & \ell_5(x_5) &= x_5. 
\end{align*}
According to Wardrop's first principle~\cite[pp.~31]{MP:15}, transportation 
networks operate at a condition where travelers select their path to minimize 
their travel time to their destination. Mathematically, Wardrop's equilibrium is 
the optimizer of the following optimization problem:
\begin{align}\label{eq:traffic_assignment}
\min_{x \in \real^{\vert \mc E\vert}} ~~~ & ~~~
\sum_{i \in \mc E} \int_0^{x_i} \ell_i(s) ds \nonumber\\
\text{subject to:} & 
\sum_{j \in \mc E:j^- =v} x_j - \sum_{j \in \mc E:j^+ = v} x_j = \delta_v(\theta(t)), \ \forall v \in \mc V, \nonumber\\
& x_i \geq 0, \ \forall i \in \mc E,
\end{align}
where $\delta_v(\theta)$ for $v \in \mc V$ is defined as $\delta_o(\theta) = \theta$ for the origin node, $\delta_d(\theta) = -\theta$ for the destination node, and $\delta_v(\theta) = 0$ for all other nodes.
The loss function in~\eqref{eq:traffic_assignment} is used to model 
travelers who will switch to a different path if it has a shorter travel time to 
destination, while the first constraint in~\eqref{eq:traffic_assignment} 
describes the network topology, namely, that traffic flows are conserved at 
each node. 
Notice that~\eqref{eq:traffic_assignment} is a time-varying optimization 
problem, where the temporal variability originates from the dependence of the 
constraint on $\theta(t)$, which describes the inflow of vehicles at the origin and outflow at the destination, measured in vehicles per hour.
For our simulations, we assume that the network inflow is sinusoidal:
\begin{align}\label{eq:lambda_inflow_model}
\theta(t) = \theta_0 - \theta_1 \cos(\omega_1 t + \phi_1) - 
\theta_2 \cos(\omega_2 t + \phi_2),
\end{align}
where $\theta_0, \theta_1, \theta_2, \omega_1, \omega_2, \phi_1, \phi_2 \in \realpos,$ satisfy 
$\theta_0 > \theta_1, $ $\theta_0 > \theta_2,$ and $\omega_2>\omega_1.$
The model~\eqref{eq:lambda_inflow_model} states that the network inflow is the 
sum of a constant term, $\theta_0,$ a slowly-varying sinusoid with angular 
frequency $\omega_1$ and a quickly-varying sinusoid with angular frequency 
$\omega_2.$ The low-frequency sinusoid is used here to describe slowly-varying 
(e.g., hourly) traffic demands, while the high-frequency sinusoid is used to 
model sudden (e.g., at the minute level) variations in traffic demand. For our 
simulations, we let $\theta_0=3$ veh/h, $\theta_1 = 1$ veh/h, 
$\theta_2 = 0.1$ veh/h, $\omega_1 = 0.1$ rad/hour, $\omega_2 = \sqrt{50}$ rad/hour, and $\phi_1=\phi_2=0,$ see Fig.~\ref{fig:wardrop_simulation} (top).

We applied Algorithm~\ref{alg:grad_feedback_design} to derive an optimization 
algorithm to solve the traffic assignment problem~\eqref{eq:traffic_assignment}.
For the synthesis, we utilized the internal model $s(z)=Sz,$ where
\[
    S = \text{diag}\biggl(\bmat{0 & 1 \\ -\omega_1^2 & 0}, \bmat{0 & 1 \\ -\omega_2^2 & 0}, 0\biggr).
\]
Notice that knowledge of $\theta_0,\theta_1, \theta_2, \phi_1, \phi_2$ is 
not required to synthesize the optimization algorithm---only the frequencies 
$\omega_1, \omega_2$ are required to be known. 
To seek a solution to the constrained problem~\eqref{eq:traffic_assignment},
consider the Lagrangian function:
\begin{multline*}
L(x,\lambda,\theta(t)) := 
\sum_{i \in \mc E} \int_0^{x_i} \ell_i(s) ds  \\
+ \sum_{v \in \mc V} \lambda_v \biggl(\delta_v(\theta(t)) - 
\sum_{j \in \mc E: j^- =v} x_j + \sum_{j \in \mc E: j^+ = v} x_j \biggr),
\end{multline*}
where $\lambda:=(\lambda_1, \dots ,\lambda_{\vert \mc V\vert})$ is the vector 
of Lagrange multipliers. We applied Algorithm~\ref{alg:grad_feedback_design} 
to the optimization problem~\eqref{eq:optimization_objective_f} with 
$f(\tilde x,\theta) = L(x, \lambda, \theta)$ (see 
Section~\ref{sec:extensions}); the inequality constraints 
in~\eqref{eq:traffic_assignment} have been accounted for by projecting $x(t)$ 
onto the feasible set. Here, matrix $L$ has been chosen so that the 
eigenvalues of $S-LQ$ are uniformly distributed in the open real interval 
$(-1,-2)$. 
Notice that, since the latencies are 
strictly increasing, the Lagrangian is strongly convex-strongly 
concave~\cite{MP:15}, and thus the problem admits a unique critical point that 
is a saddle point. It follows that our algorithm is guaranteed to converge to 
a minimizer of~\eqref{eq:traffic_assignment}. 
Simulations for this problem are presented in Fig.~\ref{fig:wardrop_simulation}.
The simulation shows that $z(t) \rightarrow \theta(t)$ as $t\rightarrow\infty$
(see Fig.~\ref{fig:wardrop_simulation}-second figure) and 
$y(t) \rightarrow 0$ (up to machine precision)
which implies that the algorithm successfully computes 
a critical point (see Fig.~\ref{fig:wardrop_simulation}-third figure). 
Finally, the bottom figure of Fig.~\ref{fig:wardrop_simulation} illustrates 
the rate of convergence of the algorithm, which, as expected, is governed by 
the placement of the eigenvalues of the observer.

\section{Conclusions}
\label{sec:conclusions}
We showed that the problem of designing algorithms for 
time-varying 
optimization problems can be reformulated as a nonlinear multivariate output 
regulation problem. This connection allowed us to prove the 
internal model principle of time-varying optimization, which states that 
(when the algorithm has access only to function evaluations of the gradient) 
asymptotic tracking can be achieved only if the algorithm  incorporates a 
reduplicated model of the temporal variability of the 
problem. On the other hand, when the time-varying parameters are 
measurable, asymptotic tracking can be achieved by seeking a mapping that 
zeros the gradient.
Importantly, our results show that asymptotic tracking can be achieved under 
more relaxed assumptions 
than what is normally imposed in the literature.
Moreover, our algorithm structure is novel in 
the literature, and it relies on the use of an observer for the temporal 
variability of the problem. This work opens the opportunity for several  
directions of future work, including the relaxation of the convexity and 
smoothness assumptions, an investigation of discrete-time problems, and 
applications to feedback optimization.

\appendix


%
%

We now summarize relevant facts in center manifold theory 
from~\cite{JC:81}; see also~\cite{SW:23}. Consider the nonlinear system:
\begin{align}\label{eq:nonlinear_sys_preliminaries}
\dot{x}=f(x),
\end{align}
where $f$ is a $C^k$ vector field defined on an open subset $U$ of $\real^n$, 
and let $x_{\circ} \in U$ be an equilibrium point for $f$, i.e., $f(x_{\circ})=0$. Without loss of generality, suppose $x^{\circ}=0$. Let
$F=\left[\frac{\partial f}{\partial x}\right]_{x=0},$
denote the Jacobian matrix of $f$ at $x=0$.
Suppose the matrix $F$ has $n^{\circ}$ eigenvalues with zero real part, 
$n^{-}$ eigenvalues with negative real part, and $n^{+}$ eigenvalues with 
positive real part. 
Let $E^{-}, E^{\circ},$ and $E^{+}$ be the (generalized) real eigenspaces of 
$F$ associated with eigenvalues of~$F$ lying on the open left half plane, the 
imaginary axis, and the open right half plane, respectively.
Note that $E^{\circ}, E^{-}, E^{+}$ have dimension  
$n^{\circ}, n^{-}, n^{+}$, respectively and that each of these spaces is 
invariant under the flow of $\dot x = Fx.$
If the linear mapping $F$ is viewed as a representation of the differential 
(at $x=0$) of the nonlinear mapping $f$, 
its domain is the tangent space $T_0 U$ to $U$ at $x=0$, and the three 
subspaces in question can be viewed as subspaces of $T_0 U$ satisfying
$T_0 U=E^{\circ} \oplus E^{-} \oplus E^{+}.$
We refer to~\cite[Sec.~A.II]{MI:01} for a precise definition of $C^k$ 
manifolds; loosely speaking, a set $S \subset U$ is a $C^k$ manifold it can be 
locally represented as the graph of a $C^k$ function.


\begin{definition}[\textbf{\textit{Locally invariant manifold}}]
A $C^k$ manifold $S$ of $U$ is 
locally invariant for~\eqref{eq:nonlinear_sys_preliminaries} if, for 
each $x_{\circ} \in S$, there exists $t_1<0<t_2$ such that the integral curve 
$x(t)$ of~\eqref{eq:nonlinear_sys_preliminaries}, such that $x(0)=x_{\circ}$, satisfies 
$x(t) \in S$ for all $t \in (t_1, t_2)$.
\QEDB\end{definition}

Intuitively, by letting $x=(y,\theta)$ and expressing~\eqref{eq:nonlinear_sys_preliminaries} 
as:
\begin{align}\label{eq:nonlinear_sys_preliminaries_b}
\dot{y}=f_y(\theta,y), && \dot{\theta}=f_\theta(\theta,y),
\end{align}
a curve $y = \pi(\theta)$ is an invariant manifold 
for~\eqref{eq:nonlinear_sys_preliminaries_b} if the solution 
of~\eqref{eq:nonlinear_sys_preliminaries_b} with 
$\theta(0)=\theta_\circ$ and $y(0) = \pi(\theta_\circ)$ lies on the curve 
$y=\pi(\theta)$ for $t$ in a neighborhood of $0.$
The notion of invariant manifold is useful as, under certain assumptions, it 
allows us to reduce the analysis of~\eqref{eq:nonlinear_sys_preliminaries} to 
the study of a reduced system in the variable $\theta$ only. The remainder of 
this section is devoted to formalizing this fact.

\begin{definition}[\textbf{\textit{Center manifold}}]
Let $x=0$ be an equilibrium of~\eqref{eq:nonlinear_sys_preliminaries}. A 
manifold $S$, passing through $x=0$, is said to be a center manifold 
for~\eqref{eq:nonlinear_sys_preliminaries} at $x=0$ if it is locally invariant 
and the tangent space to $S$ at 0 is exactly $E^{\circ}$.
\QEDB\end{definition}

Intuitively, the invariant manifold $y=\pi(\theta)$ is a center 
manifold for~\eqref{eq:nonlinear_sys_preliminaries_b} when all orbits of $y$ decay to zero and those of $\theta$ neither 
decay nor grow exponentially. 


In what follows, we will assume that all eigenvalues of $F$ have nonpositive 
real part, i.e., $n^+ = 0$.
When this holds, it is always possible to choose coordinates in 
$U$ such that~\eqref{eq:nonlinear_sys_preliminaries} is
\begin{subequations} \label{eq:preliminaries_first_order_system}
\begin{align}
& \dot{y}=A y+g(y, \theta), \label{eq:preliminaries_first_order_system_a}\\
& \dot{\theta}=B \theta+h(y, \theta), \label{eq:preliminaries_first_order_system_b}
\end{align}
\end{subequations}
where $A$ is an $n^{-} \times n^{-}$ matrix having all eigenvalues 
with negative real part, $B$ is an $n^{\circ} \times n^{\circ}$ 
matrix having all eigenvalues with zero real part, and the functions $g$ and 
$h$ are $C^k$ functions vanishing at $(y, \theta)=(0,0),$ together with all 
their first-order derivatives.
Because of their equivalence, any conclusion 
drawn for~\eqref{eq:preliminaries_first_order_system} will apply also 
to~\eqref{eq:nonlinear_sys_preliminaries}. 
The following result ensures the existence of a center manifold.

\begin{theorem}[\textbf{\textit{Center manifold existence theorem}}]
\label{thm:existence_center_manifold}
Assume that $n^+ = 0$. There exists a neighborhood 
$V \subset \real^{n^0}$ of $0$ and a class $C^{k-1}$ 
mapping $\pi: V \rightarrow \real^{n^{-}}$ such that the set
\[
    S = \{(y, \theta) \in \real^{n^{-}}  \times V: y=\pi(\theta)\},
\]
is a center manifold for~\eqref{eq:preliminaries_first_order_system}. 
\QEDB\end{theorem}


Some important observations are in order.
By definition, a center manifold 
for~\eqref{eq:preliminaries_first_order_system} passes through $(0,0)$ and 
is tangent to the subset of points whose $y$ coordinate is zero. 
Namely,
\begin{align}\label{eq:mapping_center_manifold}
    \pi(0)=0 \qquad\text{and}\qquad \frac{\partial \pi}{\partial\theta}(0)=0.
\end{align}
Moreover, this manifold is locally invariant 
for~\eqref{eq:preliminaries_first_order_system}: this imposes on the 
mapping $\pi$ the constraint:
%
%
\begin{align}\label{eq:differential_equation_center_manifold}
\frac{\partial \pi}{\partial \theta}(B \theta+h(\pi(\theta), \theta))=A \pi(\theta)+g(\pi(\theta), \theta),    
\end{align}
as deduced by differentiating with respect to time any solution 
$(y(t), \theta(t))$ of~\eqref{eq:preliminaries_first_order_system} on the manifold $y(t)=\pi(\theta(t))$. In other words, any 
center manifold for~\eqref{eq:preliminaries_first_order_system} can 
equivalently be described as the graph of a mapping $y=\pi(\theta)$ satisfying 
the 
partial differential equation~\eqref{eq:differential_equation_center_manifold}, 
with the constraints~\eqref{eq:mapping_center_manifold}.

\begin{remark}
Theorem~\ref{thm:existence_center_manifold} shows existence but 
not the uniqueness of a center manifold. Moreover, (i)
if $g$ and $h$ are $C^k, k \in \naturalpos,$  then
\eqref{eq:preliminaries_first_order_system} admits a $C^{k-1}$
center manifold; (ii) if $g$ and $h$ are 
$C^\infty$ functions, then \eqref{eq:preliminaries_first_order_system} has a $C^k$ 
center manifold for any finite $k$, but not necessarily a $C^\infty$ center 
manifold.
\QEDB \end{remark}

The next result shows that any $y$-trajectory
of~\eqref{eq:preliminaries_first_order_system}, starting sufficiently close to 
the origin converges, as time tends to infinity, to a trajectory that belongs 
to the center manifold.

\begin{theorem}
\label{thm:local_attractivity_manifold}
Assume that $n^+ = 0$ and 
suppose $y=\pi(\theta)$ is a center manifold 
for~\eqref{eq:preliminaries_first_order_system} at $(0,0)$. Let 
$(y(t), \theta(t))$ be a solution
of~\eqref{eq:preliminaries_first_order_system}. There 
exists a neighborhood $U^{\circ}$ of $(0,0)$ and real numbers $M>0$ and $K>0$ 
such that, if $(y(0), \theta(0)) \in U^{\circ}$, then for all $t \geq 0$,
\begin{align*}
\|y(t)-\pi(\theta(t))\| \leq M e^{-K t}\|y(0)-\pi(\theta(0))\|. \tag*{\QEDB}
\end{align*}
\end{theorem}


From the above discussion, any trajectory 
of~\eqref{eq:preliminaries_first_order_system} starting at a point 
$y^{\circ}=\pi\left(\theta^{\circ}\right)$ of a center manifold 
satisfies $y(t) = \pi(\zeta(t))$ and $\theta(t) = \zeta(t)$,
where $\zeta(t)$ is any solution of
\begin{align}\label{eq:zeta_dynamics}
\dot{\zeta}=B \zeta+h(\pi(\zeta), \zeta), \qquad \zeta(0) = \theta_\circ.
\end{align}
This decomposition allows us to predict the asymptotic behavior  
of~\eqref{eq:preliminaries_first_order_system} by studying the asymptotic 
behavior of the reduced-order system \eqref{eq:zeta_dynamics}. This 
is formalized in the following result.

\begin{theorem}
Suppose $\zeta=0$ is a stable (respectively, asymptotically stable, unstable) 
equilibrium of~\eqref{eq:zeta_dynamics}. Then $(y, \theta)=(0,0)$ is a stable 
(respectively, asymptotically stable, unstable) equilibrium 
of~\eqref{eq:preliminaries_first_order_system}.
\QEDB\end{theorem}

\bibliographystyle{IEEEtran}
\bibliography{BIB/alias,BIB/full_GB,BIB/GB,BIB/references}

\end{document}